\documentclass [a4paper, 11pt, reqno]{amsart}
\usepackage [latin1]{inputenc}
\usepackage {a4}
\usepackage{amscd}
\usepackage{epsfig}
\usepackage{amssymb}
\usepackage{amsmath}
\usepackage{amsthm}
\usepackage[T1]{fontenc}
\usepackage{ae,aecompl}
\usepackage[arrow, matrix, curve]{xy}
\usepackage{tikz}

\usepackage{color}

\usepackage{geometry}
\usepackage{xcolor}

\newcommand{\R} {\ensuremath{\mathbb{R}}}

\newcommand{\C} {\ensuremath{\mathbb{C}}}

\newcommand{\Ok}{\mathcal{O}}
\newcommand{\Pj}{\mathcal{P}}
\newcommand{\OO}{\mathcal{O}}
\renewcommand{\o}[1]{\overline{#1}}

\newcommand{\bomega}{{\scalebox{1.3}{$\omega$}}}
\newcommand{\w}{{\wedge}}

\newcommand{\K}{{\mathcal K}}
\newcommand{\A}{{\mathcal A}}
\newcommand{\dq}{\bar\partial}
\newcommand{\dbar}{\dq}

\DeclareMathOperator{\Jac}{Jac}
\DeclareMathOperator{\Dom}{Dom}

\newcommand{\rr}{r}

\newtheorem {satz} {Satz} [section]
\newtheorem {lem} [satz] {Lemma}
\newtheorem {cor} [satz] {Corollary}

\newtheorem {prop} [satz] {Proposition}

\newtheorem {thm} [satz] {Theorem}

\theoremstyle{remark}
\newtheorem{preremark}[satz]{Remark}
\newtheorem{preex}[satz]{Example}

\newenvironment{remark}{\begin{preremark}}{\qed\end{preremark}}
\newenvironment{ex}{\begin{preex}}{\qed\end{preex}}

\numberwithin{equation}{section}

\DeclareMathOperator{\supp}{supp}

\renewcommand{\theta}{\vartheta}

\title{Estimates for the $\dbar$-equation on canonical surfaces}

\author{M. Andersson}

\author{R. L\"ark\"ang}

\author{J. Ruppenthal}

\author{H. Samuelsson Kalm}

\author{E. Wulcan}

\address{M.\ Andersson, H.\ Samuelsson Kalm, R.\ L\"ark\"ang, E.\
  Wulcan \\ Department of Mathematical Sciences\\Chalmers University of Technology and the University of Gothenburg\\S-412 96
Gothenburg\\Sweden}
\email{matsa@chalmers.se, hasam@chalmers.se, larkang@chalmers.se,
  wulcan@chalmers.se}

\address{J.\ Ruppenthal, Department of Mathematics, University of Wuppertal, Gau{\ss}str. 20, 42119 Wuppertal, Germany.}
\email{ruppenthal@uni-wuppertal.de}

\date{\today}



\begin{document}

\begin{abstract}
We study the solvability in $L^p$ of the $\dbar$-equation in a neighborhood  of a canonical singularity
on a complex surface, a so-called du~Val singularity.  We get a quite complete picture
in case $p=2$ for two natural closed extensions $\dbar_s$ and $\dbar_w$ of $\dbar$.  For $\dbar_s$ we have solvability,
whereas for $\dbar_w$ there is solvability if and only if a certain boundary condition $(*)$ is fulfilled at
the singularity.  Our main tool is certain integral operators for solving $\dbar$
introduced by the first and fourth author, and we study mapping properties of these operators at the singularity.
\end{abstract}

\maketitle

\section{Introduction}

The classical Dolbeault-Grothendieck lemma states that locally in
$\C^n$ one can solve the $\dq$-equation $\dq u = \varphi$ if $\varphi$ is a $\dq$-closed $(0,\rr)$-form or current.
One can obtain a solution $u$ by a Koppelman formula; then $u$ is obtained through multiplication of $\varphi$ with a
smooth form followed by convolution with an integrable form, the so-called Bochner-Martinelli form.
Thus one even gains some regularity; in particular, one can solve
$\dbar$ in $C^\infty$, $L^p$, $C^\alpha$, Sobolev-spaces, etc,
see, e.g.,  \cite{Ra} or \cite{LiMi}.
 On singular varieties this is not true in general.  There are smooth $\dq$-closed
forms which have no local smooth $\dq$-potentials,
see,  e.g., \cite[Beispiel~1.3.4]{RuDipl} and \cite[Example~1]{AS}.

%
Solvability of the $\dq$-equation on singular varieties has been
studied in various articles, starting with among others \cite{HP,PS},
and in recent years solvability in $L^2$ has been of particular focus,
see, e.g.,  \cite{FOV, OV, RDuke}. There are known examples where the $\dq$-equation
is not locally solvable in $L^p$, for example when $p = 1$ or $p = 2$.
On homogeneous varieties, obstructions for solvability in $L^p$ have been described explicitly in \cite{RMatZ}.

\smallskip
In this paper we study solvability in $L^p$ of the $\dbar$-equation in a neighborhood of a canonical singularity
on a complex surface. On a surface a singularity is canonical if and only if it is a
rational double point. Such points
are well-studied and have been classified a long time ago as the so-called du~Val singularities,
see, e.g.,  the survey \cite{Du}. The possible singularities are of type $A_n$, $n\geq 1$, $D_n$, $n\geq 4$, $E_6$,
$E_7$ and $E_8$, and can be realized as isolated hypersurface
singularities in $\C^3$.

Throughout the introduction, we assume that $X$ is a surface with one isolated canonical singularity.
We will further assume that $X=\{f=0\}\subset \Omega'$, where $\Omega' \subset\subset \C^3$ is an open pseudoconvex set
and $f$ is holomorphic in a neighborhood of $\Omega'$ and that $df\neq 0$ on $\{ f = 0 \}$ except at the
singular point, which we assume is $0$.

Let $\dq_{sm}$ be the $\dq$-operator on smooth $(0,\rr)$-forms which have support not intersecting the singularity
at the origin.
We will consider two extensions of $\dbar_{sm}$  as a closed operator on $L^p(X)$.
One of them is the minimal closed extension, i.e., the
strong extension $\dq_s^{(p)}$ of $\dq_{sm}$, which is the graph
closure of $\dq_{sm}$ in $L^p_{0,\rr}(X) \times L^p_{0,\rr+1}(X)$. That is,
$\varphi \in \Dom \dq_s^{(p)} \subset L^p_{0,\rr}(X)$  if and only if there is a sequence of smooth
forms $\varphi_j \in  L^p_{0,\rr}(X)$ with
$\supp \varphi_j \cap \{0\} = \emptyset$
such that
\begin{equation*}
\varphi_j \rightarrow \varphi \ \ \ \mbox{ in } \ \ L^p_{0,\rr}(X),\quad
\dq \varphi_j \rightarrow \dq \varphi \ \ \ \mbox{ in } \ \ L^p_{0,\rr+1}(X).
\end{equation*}
The other extension is the maximal closed extension, i.e., the weak $\dq$-operator $\dq_w^{(p)}$,
so that  $\varphi \in \Dom \dq_w^{(p)}\subset  L^p_{0,\rr}(X)$ if and only if
$\dq \varphi \in L^p(X)$\footnote{This is what we take as definition of $\dq_w^{(p)}$ on $X$. However, to be precise, this definition only coincides
with the maximal closed extension of $\dq_{sm}$ for $p \geq 4/3$, which is the only case of interest to us.
In general, that $\varphi$ lies in the domain of the maximal closed extension of $\dq_{sm}$ means that
$\dq \varphi|_{X_{reg}} \in L^p(X_{reg})$. When $p \geq 4/3$, it then follows that $\dq \varphi \in L^p(X)$, see \cite[Satz~4.3.3]{RuThesis}.}.
 When it is clear from the context, we will drop the superscript $(p)$ in
 $\dq^{(p)}_s$ and $\dq^{(p)}_w$.

Let $\omega_X$ be the Poincar\'e residue of $dz_1\w dz_2\w dz_3/f$. It is an intrinsic $\dbar$-closed
meromorphic $(2,0)$-form on $X$ that is holomorphic outside of $0$.
We will see below (Proposition~\ref{prop:canonical} and Corollary~\ref{qXbound}) that there is a number $2 < q(X) \leq 4$ such that $\omega_X \in L^q(X)$
for $q<q(X)$.
Let $p(X)$ be the dual exponent of $q(X)$ and let
$$
\hat p(X)=\frac{4p(X)}{4-p(X)}.
$$
Notice that $4/3 \leq p(X)<2$ and $2 \leq \hat p(X)<4$.
For precise definitions of $L^p$-forms and $C^\alpha$-forms
on $X$, see Section~\ref{ssec:lp-forms}.

In our results, we have the following condition:

\noindent \textit{If $\varphi$ is a $(0,1)$-form in $\Dom \dbar_w^{(p)}$, where $p(X) < p \leq \infty$,
then it is said to satisfy the condition $(*)$ if
\begin{equation} \tag*{$(*)$}
\lim_{k \to \infty} \int_X \omega_X \wedge \dbar \chi_k \wedge \varphi = 0
\end{equation}
for some sequence of cut-off functions $\{\chi_k\}_k$, where each $\chi_k$ is $1$ in a neighborhood of $0$
and the support of $\chi_k$ approaches $\{0\}$ when $k \to \infty$.}

This condition is independent of the sequence of cut-off functions, see Section~\ref{ssec:star}, and is thus a
kind of boundary condition at $\{0\}$.
If $\varphi$ is $\dbar$-closed, as in the following theorem, by Stokes' theorem the condition $(*)$ means that
\begin{equation} \label{star-intro}
\int_X \omega_X\w \dbar\chi \wedge \varphi=0 
\end{equation}
for some smooth cutoff function $\chi$ that is $1$ in a neighborhood of $0$.

\begin{thm}\label{thmA}
Let $X$ be a surface as above with an isolated canonical singularity at $0$.

\smallskip
\noindent (i)   \    Assume that  $p(X)<p\le 4$.  If $\varphi$ is a $\dbar_s$-closed $(0,\rr)$-form in $L^p(X)$, $\rr=1,2$,
then there is $u$ in the domain of $\dbar_s^{(p)}$ in a neighborhood of
$0$ such that $\dbar_s u=\varphi$.

\smallskip
\noindent (ii)  \    Assume that $\hat p(X)<p\le\infty$.  If  $\varphi$ is a
$\dbar_w$-closed $(0,1)$-form in $L^p(X)$, then
there is a solution in $L^p$ to  $\dbar_w u=\varphi$ in a neighborhood of $0$.
If $p=\infty$, then
one can choose $u$ in $C^\alpha$ for  $\alpha< 4/p(X)-2$.
If $\varphi$ is a $(0,2)$-form the same holds for $p(X) < p \leq \infty$.

\smallskip
\noindent (iii)
Assume that $p(X) < p\le \hat p(X)$. If $\varphi$ is a $\dbar_w$-closed $(0,1)$-form
in $L^p(X)$, then
there is a solution in $L^p$ to $\dbar_w u=\varphi$ in a neighborhood of $0$ if and only if $\varphi$ satisfies the condition $(*)$.
\end{thm}

Notice that if $\dbar_w u=\varphi$, then \eqref{star-intro}  follows from Stokes' theorem since
$\omega_X\w\dbar\chi$ is a $\dbar$-closed smooth form with compact support.
Thus the condition  $(*)$ is necessary in the theorem.
It turns out that $(*)$ is automatically fulfilled when $\hat p(X) < p \leq \infty$, see the comment after the
proof of Theorem~\ref{thm:main3}.
In Section \ref{sec:examples} we study the condition $(*)$
explicitly for the various types of canonical singularities.
Theorem ~\ref{thm:An} asserts that in the case of a singularity of
type $A_n$, $n\geq1$,  any form $\varphi \in \Dom\dq_w \subset
L^2_{0,\rr}(X)$ satisfies  $(*)$.
For each of the other singularities, that is, of type $D_n$, $n\geq 4$, $E_6$, $E_7$ and $E_8$,
however, there is a  $(0,1)$-form $\varphi\in \ker \dq_w\subset L^2(X)$ such that
the equation $\dq_w u =\varphi$ has no solution in a neighborhood of $0$, see Theorem ~\ref{thm:examples2}.
It follows that for these $\varphi$ the condition  $(*)$ is not satisfied.

\medskip

To the best of our knowledge, the only known cases of Theorem~\ref{thmA} for general
surfaces with canonical singularities are the following:
Part (i) for $p=2$ was proven in \cite[Corollary ~1.3]{RSerre}.
Part (ii) for $p=2$ and $(0,2)$-forms was proven in \cite[Theorem~4.3]{OR}, which builds on
the vanishing result from \cite{S}.
Some weaker versions of part (ii) are known as well. For $\varphi$ with compact support, it was proven
that one can find solutions in $L^p$ (for arbitrary $p$) or with $C^\alpha$-estimates
in \cite{RZ2}. Moreover, for continuous $(0,1)$-forms $\varphi$ with compact support, $C^\alpha$-estimates
for solutions were obtained in \cite{AZ1,AZ2}.

Various results are known for the $A_1$-singularity, as is detailed in the introduction of \cite{LR}.
That there are obstructions to solving $\dbar_w$ in $L^2$ on the $D_4$-singularity was proven in \cite[Proposition 4.13]{P}.

\medskip

As mentioned above, a large part of the study of the $\dq$-equation on singular varieties
has been restricted to $L^2$-spaces.  Integral formulas open up
for new results about solvability  in $L^p$-spaces for $p \neq 2$, as well as
other norms.  For the proof of Theorem~\ref{thmA} our main tool is an
integral operator introduced in \cite{AS1,AS}.
Keeping the notation above, let $\Omega\subset\subset \Omega'$ be an open
set containing $0$ and let $D=X\cap\Omega$.
There is an operator $\K\colon C^\infty_{0,\rr}(X)\to C^\infty_{0,\rr-1}(D\setminus\{0\})$
$\rr=1,2$, such that
\begin{equation}\label{bas}
\varphi=\dbar \K\varphi+\K\dbar\varphi
\end{equation}
on $D\setminus\{0\}$.
The operator is given by an intrinsic integral kernel $K(\zeta,z)$ on
$X\times D\setminus\{0\}$ that contains the Poincar{\'e} residue $\omega_X$ as a factor in the first variable.
In \cite{AS} it was proved that  $\K$ and \eqref{bas} can be extended to certain fine sheaves
$\A_X^\rr$ of currents defined across $0$ and coinciding with $C^\infty_{0,\rr}$ outside $0$,  so that
$\dbar u=\varphi$ is solvable in $\A_X$ as soon as $\dbar\varphi=0$.

In order to prove Theorem~\ref{thmA}  we have to extend $\K$ and
\eqref{bas} to $L^p$.  To this end we first consider mapping properties of $\K$.

\begin{thm}\label{thm:main1}
 The integral operator $\K$
extends to  compact operators
\begin{equation}\label{apa1}
    \K \colon L^p_{0,\rr}(X) \to L^p_{0,\rr-1}(D), \ \
p(X) < p <\infty
\end{equation}
and
\begin{equation}\label{apa11}
    \K \colon L^\infty_{0,\rr}(X) \to C^\alpha_{0,\rr-1}(D), \ \ 0\le \alpha< 4/p(X)-2.
\end{equation}
 \end{thm}

Since the sheaves $\A_X^\rr$ are quite implicitly defined and its sections must have singularities
at $X_{sing}$ in general, it is interesting to note the following consequence of \eqref{apa11}.

\begin{cor} \label{cor:asheaves} For $X$ as above we have that
$$
\mathcal{A}_X^\rr \subset C^\alpha_{X,0,\rr}, \ \ \ \quad 0 \leq \alpha < 4/p(X)-2.
$$
\end{cor}

In order to obtain  solutions to the $\dbar_s$-equation in $L^p$ we extend \eqref{bas} by approximating
$\varphi$ by smooth forms with support away from $0$. If $\varphi$ is in the domain of
$\dbar_s^{(p)}$, it follows that \eqref{bas} holds, so if  $\dbar\varphi=0$ we
get  the solution $u=\K\varphi$ to $\dbar u=\varphi$. The problem is to see that $u$
is in the domain of $\dbar_s^{(p)}$.  
This is ``harder'' for large $p$ and our upper bound is $4$.

\begin{thm} \label{thm:main2}
Assume that $p(X)<p\le 4$.   If $\varphi \in \Dom\dq_s^{(p)} \subset L^p_{0,\rr}(X)$,
then $\mathcal{K}\varphi \in \Dom\dq_s^{(p)}$ and
\begin{equation}\label{koppla}
      \varphi(z) = \dq_s \mathcal{K} \varphi(z) + \mathcal{K} \dq_s \varphi(z), \quad
\rr=1,2.
\end{equation}
\end{thm}

In case of  $\dbar_w$  we have basically the opposite problem. Since a~priori we have no
approximation by smooth forms with support away from the origin it is "harder"
to obtain the extension of \eqref{bas} for small $p$, while it then directly follows from Theorem~\ref{thm:main1} that
the solution is in the domain of $\dbar_w$.

\begin{thm}\label{thm:main3}
Assume that $p(X)< p\le \infty$.
If  $\varphi \in \Dom\dq_w^{(p)} \subset L^p_{0,2}(X)$,  then $\mathcal{K}\varphi \in \Dom\dq_w^{(p)}$
and
\begin{equation}\label{hoppla}
      \varphi(z) = \dq_w \mathcal{K} \varphi(z) + \mathcal{K} \dq_w \varphi(z).
  \end{equation}
The same holds for $\varphi \in \Dom\dq_w^{(p)} \subset L^p_{0,1}(X)$ if
$\hat p(X) < p \leq \infty$.
If $p(X)< p\le \hat p(X)$, and
in addition $\varphi$ satisfies the condition $(*)$, then the same
conclusion holds.
\end{thm}

Notice that Theorem~\ref{thmA} follows from Theorems~\ref{thm:main1},
\ref{thm:main2} and \ref{thm:main3} and the discussion about the necessity
of the condition $(*)$ after the theorem.


Notice that if $\varphi$ is a $\dbar$-closed $(0,1)$-form with compact support then it
automatically satisfies $(*)$, and so we can solve $\dbar_w u=\varphi$ in $L^p$ if
$p(X)<p\leq \infty$. By means of a slight variation of the operator $\K$, introduced in \cite{AS1},
one can even get a solution with compact support.  In case $\varphi$ is a $(0,2)$-form in $L^p(X)$
with compact support and $\hat p(X) < p \leq \infty$, then there is a solution with compact support if and only if
\begin{equation}\label{kokett}
    \int_X \varphi\w\ h\omega_X=0  \quad \text{ for all } h\in \Ok(X),
\end{equation}
see Theorem \ref{kompakt2} below.

\bigskip

Our interest in canonical singularities is partly motivated by the earlier works \cite{LR,LR2},
where similar results as above are studied for affine cones over projective complete intersections.
The results about solvability in $L^p$ obtained in these articles are in case the degree of these homogeneous varieties is small enough.
Here, it is interesting to note that the degree is small if the singularities are mild in the sense of the minimal model program.
It turned out that positive results about solvability in $L^2$ hold precisely
for the varieties with canonical singularities, see \cite{LR2}.

The results in this article overlap with results from \cite{LR,LR2} only in the case of the $A_1$-singularity,
where in \cite{LR,LR2}, it was shown that the $\dbar_w$- and $\dbar_s$-equations are locally solvable in $L^p$ unconditionally for 
$p$ in certain intervals. 
On a general canonical surface, as studied in this article, solvability depends on the condition $(*)$.
The main novelty is the understanding of this condition and a quite sharp non-trivial estimate of the integral kernels
from \cite{AS} on such a surface. The final estimate of the integral operators is done along the same
lines as in \cite{LR,LR2}.

\medskip

We now consider the case of functions. There is an integral operator $\Pj\colon C^\infty_{0,0}(X)\to\Ok(D)$ in
\cite{AS1,AS} such that
\begin{equation}\label{trottel}
\varphi=\K\dbar\varphi+\Pj\varphi
\end{equation}
on $D\setminus\{0\}$.
In order to formulate the following result about extension of \eqref{trottel}
to $L^p$ we need a condition $(*)$ for  functions $\varphi$ that is explained in 
Section~\ref{gorilla} below. 


\begin{thm}\label{functions}
    Let $X$ be as above. Then the operator $\mathcal{P}$ extends to a compact operator $\mathcal{P} : L^p_{0,0}(X) \to \Ok(D)$
    for $1 \leq p \leq \infty$. If $\varphi \in \Dom \dbar^{(p)}_s \subseteq L^p_{0,0}(X)$ where $p(X) < p \leq 4$, then
    \eqref{trottel} holds.
    This equality also holds if $\varphi \in \Dom \dbar^{(p)}_w \subseteq L^p_{0,0}(X)$ and either $\hat{p}(X) <p \leq \infty$
    or $p(X) < p \leq \hat{p}(X)$ and $\varphi$ satisfies the condition $(*)$. 
\end{thm}

\smallskip

The present paper is organised as follows. After providing some preliminaries in
Section~\ref{sec:preliminaries}, in Section~\ref{sec:main} we recall the integral
formulas from \cite{AS1, AS},
analyse their integral kernels and prove Theorem ~\ref{thm:main1} and its corollary.
Section ~\ref{sec:homotopy-formulas} is devoted to $\dq$-homotopy formulas and proofs
of Theorems ~\ref{thm:main2},  ~\ref{thm:main3} and ~\ref{functions} and also
to a discussion of condition $(*)$. 
We also include a discussion of the domain of the $\dbar_X$-operator from
\cite{AS} and prove that $\mathcal K\varphi\in \Dom \dbar_X$ for
certain $\varphi\in \Dom \dbar_s$, see Theorem \ref{thm:main5}.
In Section ~\ref{sec:examples} we characterize the du~Val singularities with respect to $(*)$.
Finally 
we recall some integral estimates on singular varieties from \cite{LR2} in an appendix, Section ~\ref{sect:basic-estimates}.

\section{Preliminaries}\label{sec:preliminaries}

In this section we specify the spaces of differential forms that we consider
and explain some basic tools. Throughout the section
$i\colon X \hookrightarrow\Omega'\subset \C^N$ is
an analytic variety of pure dimension $n$, and $D \subset\subset X$
is an open subset of $X$.


\subsection{$C^\alpha$- and $L^p$-forms on an analytic variety}\label{ssec:lp-forms}
%
Let $1\leq p\leq \infty$. Since $D^* := D \cap X_{reg}$ is a submanifold of
some open subset of $\C^N$, it inherits a Hermitian metric from $\C^N$.
We say that a $(0,\rr)$-form $\varphi$ is in $L^p_{0,\rr}(D)$ if $\varphi|_{D^*}$ is in $L^p_{0,\rr}(D^*)$ with respect
to the induced volume form $dV_X$.
When it is clear from the context, we will drop the subscript in $L^p_{0,\rr}(D)$.

It will be convenient to represent  $(0,\rr)$-forms on $X$ in a certain ``minimal'' manner:
Any $(0,\rr)$-form $\varphi$ on $D^*$ can be written (uniquely) in the form
\begin{equation}\label{apa}
    \varphi = \sum_{|I|=\rr} \varphi_I d\bar{z}_I,
\end{equation}
where $d\bar z_I=d\bar z_{i_1}\wedge\cdots\wedge d\bar z_{i_\rr}$ if
$I=\{i_1,\ldots, i_\rr\}$,
such that
\begin{equation} \label{eq:minrep}
    |\varphi|^2 (z)= {2}^\rr \sum |\varphi_I|^2(z)
\end{equation}
at each point $z\in D^*$.  In fact, one starts with any representation and then at each
point takes the orthogonal projection of the form onto
$\Lambda^{0,\rr}T^*D^*$, see, e.g., \cite[Lemma~2.2.1]{RuThesis}.
In particular, then $\varphi \in L^p_{0,\rr}(D)$ if and only if $\varphi_I \in L^p(D)$ for all $I$.
If one has an arbitrary representation of $\varphi$ of the form \eqref{apa}, then
\begin{equation} \label{eq:non-minrep}
    |\varphi|^2 (z) \leq {2}^\rr \sum |\varphi_I|^2(z),
\end{equation}
and so, in general, $\varphi \in L^p_{0,\rr}(D)$ if $\varphi_I \in L^p(D)$ for all $I$.

\smallskip

Recall that a form $\varphi$ on $D$ is in $C^k(D)$, $0\le k\le\infty$, if  (locally) it is the pullback of a $C^k$-form in ambient space; i.e., there exists a representation
\eqref{apa} such that all the coefficients (locally) admit $C^k$-extensions
to a neighborhood of $D$.
For $0 \leq \alpha < 1$, we say that a $(0,\rr)$-form $\varphi$ on $D$ is
$C^\alpha$ if locally on $D$  there is a representation \eqref{apa} such that all the coefficients $\varphi_I$ are
$C^\alpha$, i.e., H\"older continuous with exponent $\alpha$, on $D$.  It is well-known,
that a function that is $C^\alpha$ on $D$ has a $C^\alpha$-extension to
ambient space, see, e.g., \cite{M}.   Thus a form $\varphi$ on $D$ is in $C^\alpha$ if and only if
it is the pull-back to $D$ of a $C^\alpha$-form in ambient space.
Notice that $C^1(D)\subset C^\alpha(D)$ for all $\alpha<1$.
%
For $\alpha = 1$, we denote the Lipschitz continuous functions by $C^{0,1}(D)$
in order to avoid conflict of notation with continuously differentiable functions.

It is not hard to check that these definitions are independent of the choice
of embedding of $X$, and hence are intrinsic notions on $X$.
Fix an embedding $D\to \Omega\subset \C^{N}$.  We can then define the
H\"older-norm
 \begin{equation}\label{apa2}
\|\varphi\|^2_\alpha=
\inf {2}^\rr \sum \|\varphi_I\|^2_{\alpha},
\end{equation}
 of a form $\varphi$ on $D$, where the infimum runs over all
representations \eqref{apa} of $\varphi$ in ambient space, and  the norms
on the right hand side  of \eqref{apa2} are over $D$.
This norm is, up to constants, independent of the embedding.


\begin{remark}
Regularity properties of $\varphi$ like smoothness, H\"older continuity etc, will be reflected
by the coefficients on $D^*$ in the minimal representation
\eqref{eq:minrep} above.  However, even if $\varphi$ is smooth across the singularity,
the coefficients in the minimal representation may be discontinuous there.
\end{remark}

Using the minimal representation \eqref{eq:minrep}, and the inequality \eqref{eq:non-minrep}
for not necessarily minimal representations, and the analogous
inequality for H\"older norms, we get the following lemma. 

\begin{lem} \label{lma:Lpformsfunctions}
If $\mathcal{K}$ is an integral operator mapping $(0,\rr)$-forms in $\zeta$ to $(0,\rr-1)$-forms in $z$,
defined by an integral kernel
\begin{equation*}
    K(\zeta,z) = \sum_{|L|=n,|I|=\rr-1,|J|={n-\rr}} K_{I,J,L}(\zeta,z) d\overline{z}_I \wedge {d\overline{\zeta}_J} \wedge d\zeta_L,
\end{equation*}
then $\mathcal{K}$ is a bounded linear map
$L_{0,\rr}^p(X) \to L_{0,\rr-1}^p(D)$ if
\begin{equation*}
    f(\zeta) \mapsto \int_{X} K_{I,J,L}(\zeta,z) f(\zeta) dV_X(\zeta)
\end{equation*}
is a bounded linear map $L^p(X) \to L^p(D)$, and a bounded linear map
$L_{0,\rr}^\infty(X) \to C_{0,\rr-1}^\alpha(D)$ if
\begin{equation*}
    f(\zeta) \mapsto \int_{X} K_{I,J,L}(\zeta,z) f(\zeta) dV_X(\zeta)
\end{equation*}
is a bounded linear map $L^\infty(X) \to C^\alpha(D)$.
\end{lem}

\subsection{Cut-off functions}\label{ssec:cut-off}

We will use the following cut-off functions to approximate forms
by forms with support away from isolated singularities. 
As in the proof of Proposition~3.3 in \cite{PS}, let $\rho_k: \R\rightarrow [0,1]$, $k\geq 1$, be smooth cut-off functions
satisfying
$$\rho_k(x)=\left\{\begin{array}{ll}
1,\  x\leq k,\\
0,\  x\geq k+1,
\end{array}\right.$$
and $|\rho_k'|\leq 2$. Moreover, let $r: \R_+\rightarrow [0,1/2]$ be a smooth increasing function such that
$$r(x)=\left\{\begin{array}{ll}
x,\  & x\leq 1/4,\\
1/2,\  & x\geq 3/4,
\end{array}\right.$$
and $|r'|\leq 1$.
As cut-off functions we will use $\mu_k(\zeta):=\rho_k\big(\log(-\log r(|\zeta|))\big)$
on $X$ if $X$ has an isolated singularity at $0$. Note that there is a
constant $C$ such that
\begin{equation}\label{eq:cutoff2}
\big| \dq \mu_k(\zeta)\big| \leq C\frac{\chi_k(|\zeta|)}{|\zeta| \big| \log|\zeta|\big|},
\end{equation}
where $\chi_k$ is the characteristic function of $[e^{-e^{k+1}}, e^{-e^k}]$.

\begin{lem}{\cite[Lemma 5.1]{LR2}}\label{lem:cut-off}
Let $\varphi\in L^p_{0,\rr}(D)$
with $\dq_w \varphi\in L^{p'}_{0,\rr+1}(D)$, where $\frac{2n}{2n-1}\leq p \leq \infty$ and $1\leq p' \leq \infty$.
Let $\varphi_k :=\mu_k \varphi$
and define $1\leq \lambda \leq 2n$ by the relation
\begin{equation}\label{defn:lambda}
\frac{1}{\lambda} = \frac{1}{p} + \frac{1}{2n}.
\end{equation}
If $\gamma=\min\{\lambda, p'\}$, then
$
\varphi_k \to \varphi \ \mbox{ in }\  L^p_{0,\rr}(D),\quad
\dq \varphi_k \rightarrow \dq_w \varphi \ \mbox{ in } \ L^\gamma_{0,\rr+1}(D).
$
\end{lem}

\subsection{On the domain of the $\dq_s$-operator}

\begin{lem}{\cite[Lemma 5.2]{LR2}}\label{lem:cut-off2}
Assume that $X$ has an isolated singularity at $0\in D$ and that $D$ has smooth boundary.
Let $1\leq p \leq 2n$
and
let $\varphi\in L^p_{0,\rr}(D)$ such that $\varphi \in \Dom \dq_w^{(p)}$.
Then $\varphi\in \Dom\dq_s^{(p)}$ if and only if there
exists a sequence of bounded forms $\varphi_j\in L^\infty_{0,\rr}(D)$, $\varphi_j\in\Dom\dq_w^{(p)}$,
such that
\begin{equation}\label{eq:appr1}
\varphi_j \to \varphi \  \ \mbox{ in } \ L^p_{0,\rr}(D),\quad
\dq_w \varphi_j \to \dq_w \varphi \  \ \mbox{ in } \ L^p_{0,\rr+1}(D).
\end{equation}
\end{lem}

\medskip
\section{Integral operators on surfaces with  canonical singularities}
\label{sec:main}

\subsection{The Koppelman integral kernels for a hypersurface}
\label{ssec:kernels}

Let us recall the definition of the Koppelman integral operators from \cite{AS}
in the situation of a hypersurface $i\colon X \subset\Omega'\subset \C^{n+1}$
defined by $X = \{ \zeta \in \Omega';\  f(\zeta) = 0 \}$, where $f$ is a holomorphic function on $\Omega'$ and $df$
is non-vanishing on $X_{reg}$, where $\Omega'$ is
pseudoconvex.
Let $\Omega \subset\subset \Omega'$ be an open set
and let $D := X \cap \Omega$.

Let $\omega_X$ be the
Poincar\'e residue of the meromorphic form
$d\zeta_1 \wedge\ldots\wedge d\zeta_{n+1}/f$.
This means that $\omega_X$ is the unique meromorphic $(n,0)$-form on $X$ such that
 \begin{equation}\label{eq:structure-form1}
df\wedge \omega_X  =2\pi i
d\zeta_1 \wedge \cdots  \wedge d\zeta_{n+1}.
\end{equation}
In \cite[Section~3]{AS} so-called structure forms were introduced as
generalizations of the Poincar\'e residue for more general $X$; we
will therefore refer to $\omega_X$ as the structure form on $X$.
Recall that $1/f$ and $\omega_X$ define  principal value currents on
$\Omega'$ and $X$, respectively. Identifying these with their respective currents,
$\omega_X$ can be defined as the unique current such that
$$
i_*\omega_X=  \dbar\frac{1}{f}\w d\zeta_1 \wedge \cdots \wedge d\zeta_{n+1}.
$$

For coordinates $\zeta=(\zeta_1,\ldots, \zeta_{n+1})$
such that $\partial f/\partial \zeta_1$ is generically non-vanishing on $X_{reg}$,
$\omega_{X}$ is the pull-back of
\begin{equation}\label{eq:structure-form0}
2\pi i\frac{d\zeta_2 \wedge \cdots \wedge d\zeta_{n+1}}{\partial f / \partial \zeta_1}
\end{equation}
to $X$.
Alternatively, letting
\begin{equation}\label{eq:structure-form2}
   \vartheta:= 2\pi i\sum_{\ell=1}^{n+1} \frac{\overline{f'_{\ell}}}{|\partial f|^2}\frac{\partial}{\partial\zeta_{\ell}},
\end{equation}
where $f'_{\ell}=\partial f/\partial \zeta_{\ell}$, we have that $\omega_X$ is realised as the pull-back to $X$
of $\vartheta\lrcorner d\zeta_1\wedge\cdots\wedge d\zeta_{n+1}$.
Here, the norm $|\partial f|$ is computed in $\C^{n+1}$, i.e., $|\partial f|^2=\sum |f_l'|^2$.

Let $\eta_j=\zeta_j-z_j$ and
let $\delta_\eta$ be interior multiplication
by $2\pi i \sum \eta_j \partial/\partial \zeta_j$.
We will consider forms with anti-holomorphic differentials
of both $\zeta$ and $z$ but only holomorphic
differentials with respect to $\zeta$.
The (full) Bochner-Martinelli form is $B := b + b\w \dq b + \dots + b\w (\dq b)^{n},$
where
\begin{equation*}
    b := \frac{1}{2\pi i}\frac{\bar\eta_1 d\zeta_1+\ldots+\bar\eta_{n+1}d\zeta_{n+1}}{|\eta|^2} = \frac{1}{2\pi i}\frac{\bar{\eta}\cdot  d\zeta}{|\eta|^2}.
\end{equation*}
Notice that
\begin{equation}\label{orm}
B_k:=b\w (\dbar b)^{k-1}=\frac{1}{(2\pi i)^k} \frac{\bar\eta\cdot d\zeta\w (d\bar\eta\w d\zeta)^{k-1}}{|\eta|^{2k}}=\Ok(1/|\eta|^{2k-1}),
\end{equation}
where $d\bar\eta\w d\zeta=d\bar \eta_1\w d\zeta_1+\cdots + d\bar \eta_{n+1}\w d\zeta_{n+1}$. 

A smooth form $g=g_{0,0}+\cdots + g_{n+1,n+1}$ in $\Omega'\times\Omega'$,  here lower indices
denote bidegree,  is a {\it weight with respect to $\Omega$} if
$(\delta_\eta-\dbar)g=0$ and $g_{0,0}(z,z)=1$ for $z\in\overline \Omega$.  We say that $g$ is
holomorphic with respect to $z$ if the coefficients are holomorphic in $z$ and
there are no anti-holomorphic differentials with respect to $z$.

\begin{ex}[Holomorphic weights with compact support]\label{viktex}
Let $\chi=\chi(\zeta)$ be a cut-off function with compact support in $\Omega'$
which is $1$ in a Stein neighborhood $\Omega''\subset\subset\Omega'$
of $\overline\Omega$. Moreover, let $s(\zeta,z) = \sum s_i(\zeta,z)
d\zeta_i$ be a $(1,0)$-form defined for $\zeta\in \supp \dbar\chi$ and
$z\in \overline \Omega$, such that $\delta_\eta s = 1$ and $s$ is
smooth in $\zeta$ and holomorphic in $z$.
%
Such an  $s$ exists since $\Omega''$ is Stein in $\Omega'$.
Then
\begin{equation*}
    g := \chi - \dq \chi \wedge \big(s+ s\w(\dq s) + \dots + s\w (\dq s)^{n}\big)
\end{equation*}
is a weight in $\Omega'\times\Omega'$ with respect to $\Omega$ that has compact support in
$\Omega_\zeta'$ and is holomorphic with respect to $z$. 
If $\Omega$ is the unit ball in $\C^{n+1}$  we can choose
\begin{equation*}
    s= \frac{\overline{\zeta} \cdot d\zeta}{2\pi i(|\zeta|^2-\bar{\zeta}\cdot z)}.
\end{equation*}
 \end{ex}

A holomorphic $(1,0)$-form $h=h_1 d\zeta_1+\cdots + h_{n+1}d\zeta_{n+1}$ in $\Omega'\times\Omega'$
is a {\it Hefer form} for $f$ if
$
\delta_\eta h=f(\zeta)-f(z).
$
 Since $h_j(\zeta,\zeta)=(2\pi i)^{-1}\partial f/\partial \zeta_j$ it follows that
\begin{equation}\label{eq:hefer0}
h (\zeta,z) = (2\pi i)^{-1} \, df(\zeta)  + O(|\eta|),
\end{equation}
where $O(|\eta|)$ is a holomorphic $1$-form with coefficients in the ideal generated by $\eta_1,\ldots,\eta_{n+1}$.

\smallskip
Let $h$ be such a Hefer form and let $g$ be a weight as in Example~\ref{viktex}. 
We can then define an integral operator  $\mathcal{K}$ that acts on  forms on $X$ and 
produces forms on $D=X\cap \Omega'$ in the following way:
We let  
\begin{equation}\label{eq:AS1}
    (\mathcal{K} \varphi)(z) = \int_{X_\zeta} K(\zeta,z) \wedge \varphi(\zeta),
\end{equation}
where the kernel has the form
\begin{equation} \label{eq:Kdef0}
    K(\zeta,z) = \omega_X(\zeta) \wedge \tilde{K}(\zeta,z),
\end{equation}
\begin{equation*}
    d\zeta_1 \wedge \cdots \wedge d\zeta_{n+1} \wedge \tilde{K}(\zeta,z) = h \wedge (g\wedge B)_n,
\end{equation*}
and $(g\wedge B)_n$ denotes the components of $g\wedge B$ of bidegree
$(n,*)$, cf.\ \cite[Section~8]{AS}.
It follows that $K(\zeta,z)=\vartheta\lrcorner \big (h\wedge (g\wedge
B)_n\big )$ and so, in view of \eqref{eq:structure-form2},
\eqref{orm}, and \eqref{eq:hefer0} we get that
\begin{eqnarray*}
K(\zeta,z) &=& \vartheta\lrcorner \left(\big(df/2\pi i + O(|\eta|)\big)\wedge \sum_ic_i(\zeta,z)\frac{\bar{\eta}_i}{|\eta|^{2n}}\right) \\
&=&
\vartheta\lrcorner \left(df/2\pi i \wedge \sum_ic_i(\zeta,z)\frac{\bar{\eta}_i}{|\eta|^{2n}} +
d\zeta_1\wedge\cdots\wedge d\zeta_{n+1}\wedge \sum_{i,j}b_{ij}(\zeta,z)\frac{\bar{\eta}_i\eta_j}{|\eta|^{2n}} \right) \\
&=&
\sum_{i,j,k} a_{ijk}(\zeta,z)\frac{\bar{\eta}_i}{|\eta|^{2n}}\frac{f'_j\overline{f'_k}}{|\partial f(\zeta)|^2} +
\omega_X(\zeta)\wedge \sum_{i,j} b_{ij}(\zeta,z)\frac{\bar{\eta}_i\eta_j}{|\eta|^{2n}},
\end{eqnarray*}
where the $c_{i}$ and the $a_{ijk}$ are smooth $(n,*)$-forms and the $b_{ij}$ are smooth $(0,*)$-forms.
We have thus shown

\begin{prop}\label{prop:Kestimate}
We can write $\mathcal{K} = \mathcal{K}_1 + \mathcal{K}_2$, where
$\mathcal{K}_1$ and $\mathcal{K}_2$ are defined by integral kernels
$k_1$ and $k_2$, respectively, that are sums of terms of the form
    \begin{equation}\label{k1}
        a(\zeta,z)\frac{\bar{\eta}_i}{|\eta|^{2n}}\frac{f'_j(\zeta)\overline{f'_k}(\zeta)}{|\partial f(\zeta)|^2},
\end{equation}
and
\begin{equation}\label{k2}
            b(\zeta,z)\wedge \omega_X(\zeta) \frac{\bar \eta_i \eta_j}{|\eta|^{2n}},
    \end{equation}
respectively,
 where $a(\zeta,z)$ and $b(\zeta,z)$ are smooth on $X \times D$.
\end{prop}

We also need to consider the projection operator $\mathcal{P}$,
which is defined by
\begin{equation}\label{eq:AS2}
    (\mathcal{P} \varphi)(z) = \int_{X_\zeta} P(\zeta,z) \wedge \varphi(\zeta),
\end{equation}
where the integral kernel $P(\zeta,z)$ is defined in a similar way to \eqref{eq:Kdef0}.
Namely,
\begin{equation*}
    P(\zeta,z) = \omega_X(\zeta) \wedge \tilde{P}(\zeta,z),
\end{equation*}
where
\begin{equation*}
    \tilde{P}(\zeta,z) \wedge d\zeta_1 \wedge \dots \wedge d\zeta_{n+1} = h \wedge g_{n,n},
\end{equation*}
cf.\ \cite[(5.5)]{AS}.
Notice that since $h$ and $g$ are smooth, $\tilde{P}$ is smooth, and so $|P(\zeta,z)|\lesssim |\omega_X(\zeta)|$.
If $X$ has an isolated singularity in $\Omega$ and we choose $g$ according to Example~\ref{viktex}, then
for each $z$, $\zeta\mapsto g_{n,n}(\zeta,z)$ is supported away from $X_{sing}$ and the corresponding $P$ is
then smooth in $\zeta$ and holomorphic in $z$.

\subsection{$L^{2+}$-property of the structure form for a canonical hypersurface}

\begin{prop}\label{prop:canonical}
Let $i\colon Y \to \Omega\subset\C^{n+1}$ be a hypersurface with canonical singularities and $X\subset\subset Y$.
Then there exists a real number $q(X)>2$ such that
$\omega_Y \in L^q(X)$
for $1\leq q < q(X)$, where $\omega_Y$ is the structure form of $Y$.
\end{prop}

\begin{proof}

We denote by $\bomega_Y^n$ Grothendieck's dualizing sheaf (sometimes also called the sheaf of Barlet-Henkin-Passare holomorphic $n$-forms on $Y$).
As $Y$ is a hypersurface,
in particular Cohen-Macaulay,  $\bomega_Y^n$ is a locally free $\Ok_Y$-module
of rank one, and
the structure form $\omega_Y$ is a generator of $\bomega_Y^n$,
see, e.g.,  \cite{AS} and \eqref{eq:structure-form0}.

Let $\pi: M \rightarrow Y$ be a resolution of singularities
such that the exceptional divisor has only normal crossings.
Since $Y$ has canonical singularities, $\pi^*\omega_Y$ extends across $\pi^{-1}Y_{sing}$
to a holomorphic $n$-form on $M$. Pick any Hermitian metric on $M$ and let $dV_M$ be the corresponding
volume form. Then $i^{n^2}\pi^*\omega_Y\wedge\pi^*\overline{\omega}_Y=AdV_M$ for some smooth non-negative
function $A$ on $M$.

Let $s$ be local coordinates on $M$ and let $dV_s=(i/2)^n ds_1\w d\bar s_1\w\ldots\w ds_n\w d\bar s_n$.
Then $dV_s=BdV_M$ for some smooth positive function $B$.
Let $\varpi=i\circ \pi$, where $i$ is the inclusion $Y\hookrightarrow \Omega\subset \C^{n+1}$. Then, on $M\setminus\pi^{-1}Y_{sing}$,  $s\mapsto \varpi(s)$ is a local parametrization of $Y_{{reg}}\subset \Omega$
and it is well-known that $\pi^* dV_Y=\det H dV_s$,
where $H=\ ^t \overline{\Jac \varpi} \cdot \Jac \varpi \geq 0$
and
$\Jac \varpi = \big( \partial \varpi_\nu / \partial s_\mu\big)_{\nu,\mu}$ is
the Jacobian matrix of $\varpi$. 
Notice that $\det  H$ is a non-negative real-analytic function that vanishes precisely
on $\pi^{-1}Y_{sing}$. It follows that $(\det H)^{-\epsilon/2}$ is locally integrable with respect to $dV_M$
for some $\epsilon>0$. We now get
\begin{equation*}
\pi^*dV_Y=\det H dV_s=\det H B dV_M=:C dV_M.
\end{equation*}
Thus $C$ is a globally defined function and each point in $Y$ has a neighborhood where $C^{-\epsilon/2}$
is integrable for some $\epsilon>0$. Since $\pi^{-1}X\subset\subset M$, there is an $\epsilon(X)>0$ such that
$C^{-\epsilon/2}$ is integrable on $\pi^{-1}X$ for all $\epsilon<\epsilon(X)$.

Recall that $|\omega_Y|^2dV_Y=i^{n^2}\omega_Y\wedge\overline{\omega}_Y$. Pulling back to $M$ we get
$\pi^*|\omega_Y|^2C dV_M=A dV_M$ and thus $\pi^*|\omega_Y|^2=AC^{-1}$. Hence
\begin{equation} \label{omegaQ}
\int_X |\omega_Y|^q dV_Y=\int_{\pi^{-1}X} A^{q/2}C^{-q/2+1} dV_M <\infty
\end{equation}
as long as $q-2<\epsilon(X)$, and so we may take $q(X)=2+\epsilon(X)$.
\end{proof}

\begin{lem}
    Let $Y \subset \Omega \subset \C^3$ be a hypersurface with an isolated canonical singularity, and let $X$ and $q(X)$ be as in Proposition~\ref{prop:canonical}.
    Then $q(X) \leq 2+\frac{2}{m}$, where $m$ is the maximum of the multiplicities of the divisors in the unreduced exceptional divisor in a minimal resolution of singularities of $Y$.
\end{lem}

\begin{proof}
    Assume that $Y = \{ f = 0 \} \subset \Omega \subset \C^3$, and that $Y$ has an isolated singularity at $z = 0$. Then we claim that on $Y_{{reg}}$
    \begin{equation} \label{omegaNorm}
        |\omega_Y| = c \frac{1}{|\partial f|},
    \end{equation}
    for some constant $c$, where as above the norm $|\omega_Y|$ is with respect to the norm on $Y_{{reg}}$ induced by the norm on $\C^3$,
    while $|\partial f|$ is with respect to the norm on $\C^3$.
    Indeed, for any $(2,0)$-form $\alpha$ on $Y_{{reg}}$, one has the formula
    \begin{equation*}
        |\alpha|_{Y_{{reg}}} = \frac{|\alpha \wedge \partial f|_{\C^3}}{|\partial f|_{\C^3}},
    \end{equation*}
    and thus \eqref{omegaNorm} follows from \eqref{eq:structure-form1}.

    Let $A$ and $C$ be as in the proof of Proposition~\ref{prop:canonical}.
    Let $\pi : M \rightarrow Y$ be a minimal resolution of singularities of $Y$. This resolution is crepant, i.e.,
    $\pi^* \bomega^2_Y = \bomega^2_M$, see for example \cite[Theorem~7.5.1]{Ishii}. Thus, the function $A$
    is strictly positive.

    Since $Y$ has an isolated singularity at $0$, $|\partial f| \lesssim |z|$, so by \eqref{omegaNorm}, $|\omega_Y| \gtrsim 1/|z|$.
    Since $A$ is strictly positive, $\pi^* |\omega_Y| \sim C^{-1/2}$, and it thus follows from \eqref{omegaQ} that for $q \geq 2$,
    \begin{equation*}
        \int_X |\omega_Y|^q dV_Y \gtrsim \int_{\pi^{-1} X} \frac{1}{\pi^* |z|^{q-2}} dV_M.
    \end{equation*}
    If $Z_i$ is an irreducible component of the unreduced exceptional divisor $Z$, and $Z_i$ has multiplicity
    $m_i$, then $\pi^* |z|^2$ vanishes to order $2m_i$ along $Z_i$, and thus, in order for the integral on the right-hand side
    to be finite, we must have that $m_i(q-2) < 2$ for all $m_i$.
\end{proof}

In combination with a calculation of the multiplicities as in for example
\cite[Example~7.2.5]{Ishii} or \cite[Proposition~3.8]{BPV},
we obtain the following corollary.

\begin{cor} \label{qXbound}
    If $Y$ is a surface with an isolated $A_n$, $D_n$, $E_6$, $E_7$ or $E_8$-singularity, and $X \subset \subset Y$,
    then $q(X)$ is at most $4$, $3$, $2+2/3$, $2+1/2$ or $2+1/3$, respectively.
\end{cor}

In particular, we always have that $q(X) \leq 4$, so $p(X) \geq 4/3$ for all surfaces with canonical singularities.

\subsection{Mapping properties of $\K$}
\label{ssec:mapping-properties}
\begin{proof} [Proof of  the $L^p$ mapping properties in Theorem~\ref{thm:main1}]
By Proposition ~\ref{prop:Kestimate}  we have the decomposition
$K(\zeta,z) = k_1(\zeta,z) + k_2(\zeta,z)$,
where
$$
|k_1(\zeta,z)| \lesssim \frac{1}{|\zeta-z|^3}, \quad\quad
k_2(\zeta,z)  = \omega_X(\zeta) \wedge \frac{b'(\zeta,z)}{|\zeta-z|^2},
$$
where $b'(\zeta,z)$ is bounded.
By Lemma~\ref{lem:estimate3},
$k_1$ is uniformly integrable over $X$ in $\zeta$ as well as in $z$,
and so $\K_1$ 
maps  $L^p(X) \to L^p(D)$ continuously
for all $1\leq p\leq \infty$ by the generalized Young inequality, \cite[Appendix B]{Ra}
and Lemma~\ref{lma:Lpformsfunctions}.

Note that we can then decompose $\K_2$ into the
consecutive application of two operators
\begin{eqnarray}\label{eq:decomp}
\varphi(\zeta) &\mapsto& \varphi(\zeta)\wedge \omega_X(\zeta)
\ \ \mapsto \ \ \int_{X} \varphi(\zeta) \wedge \omega_X(\zeta) \wedge
                         \frac{b'(\zeta, z)}{|\zeta-z|^2}.
\end{eqnarray}
To analyse this chain, choose $2<q<q(X)$ so that $\omega_X \in L^q(X)$.
By H\"older's inequality,
the operator
$\varphi \mapsto \varphi\wedge \omega_X$
maps
$L^p(X) \to L^a(X)$ continuously
for $1\leq a \leq \infty$ defined by $1/a=1/p+1/q$ (for $p$ so that $1/p+1/q \leq 1$).

The second operator in \eqref{eq:decomp}
can again be analysed by the generalised Young inequality. 
By Lemma~\ref{lem:estimate3}, $|\zeta-z|^{-2}\in L^s(X)$ in $\zeta$ and in $z$
for all $s<2$, in particular for $s$ defined by $1/s + 1/q =1$, since $q>2$.
Then, since $1/p=1/a-1/q=1/a+1/s-1$, it follows from the generalised Young inequality, \cite[Appendix B]{Ra},
that \[
\varphi\mapsto\int_X\varphi(\zeta) \wedge \frac{b'(\zeta, z)}{|\zeta-z|^2}\]
 maps $L^a(X)\to L^p(D)$ continuously.
Combining, we see that the composed operator \eqref{eq:decomp}
given by the kernel $k_2$ is a bounded mapping  $L^p(X)\to L^p(D)$
for any $p$ such that $1/p+1/q\leq 1$. Thus $\K$ is a bounded mapping
$L^p(X)\to L^p(D)$ for all $p(X)<p\leq \infty$.

The kernel $k_1$ is integrable in both variables, and by truncating it,
we get a bounded kernel corresponding to a compact operator; by standard arguments, cf., for example
\cite[Appendix C]{Ra}, this
converges to $\K_1$, and it is thus a compact operator.
If we decompose the operator $\K_2$ as in \eqref{eq:decomp}, the same holds for the right-most operator,
and thus also $\K_2$ is compact.
\end{proof}


\begin{proof} [Proof of the $C^\alpha$ mapping properties in Theorem~\ref{thm:main1}]
    Let us first consider the operator $\mathcal{K}$.
Note that for $\nu=1,2$, $k_\nu(\zeta, z)\varphi (\zeta)$ is a sum of
terms of the form $k_\nu'(\zeta, z)\varphi'(\zeta) d\zeta_I\wedge d \bar \zeta_J \wedge d
\bar z_K$
and $\K_\nu \varphi(z)$ is a sum of terms $(\K_\nu \varphi)'(z)d\bar
z_K:=\int k_\nu'(\zeta, z)\varphi'(\zeta) d\zeta_I\wedge d \bar \zeta_J \wedge d
\bar z_K$, where $k_\nu'(\zeta, z), \varphi'(\zeta)$, and $(\K_\nu
\varphi)'(z)$ are functions.
Using that
    \begin{equation*}
        |(\mathcal{K}_\nu)'\varphi(z) - (\mathcal{K}_\nu)'\varphi(w)| \lesssim \|\varphi'\|_{L^\infty} \int |k'_\nu(\zeta,z) - k'_\nu(\zeta,w)|,
    \end{equation*}
it follows that $\K_\nu$ maps into $C^\alpha$ if
    \begin{equation}\label{eq:hoelder-estimate}
        \int |k'_\nu(\zeta,z) - k'_\nu(\zeta,w)| \lesssim |z-w|^\alpha.
    \end{equation}
for each $k_\nu'$.
%
%

    For $\nu=1$, we may assume that $k_1$ is of the form
    \eqref{k1}. Then $k_1'$ is a sum of functions of the form \eqref{k1} with
    $a(\zeta, z)$ replaced by one of its coefficients $a'(\zeta, z)$. We may assume
    that $k_1'$ is one such function; then
    \begin{align*}
        &\int |k'_1(\zeta,z) - k'_1(\zeta,w)| \lesssim
        \int |a'(\zeta,z) - a'(\zeta,w)| \left|\frac{\overline{\zeta_i - z_i}}{|\zeta-z|^4}
        \frac{f'_j(\zeta)\overline{f'_k}(\zeta)}{|\partial f(\zeta)|^2} \right| + \\
        + & \int |a'(\zeta,w)| \left|\left(\frac{\overline{\zeta_i - z_i}}{|\zeta-z|^4} - \frac{\overline{\zeta_i - w_i}}{|\zeta-w|^4}\right)
        \frac{f'_j(\zeta)\overline{f'_k}(\zeta)}{|\partial f(\zeta)|^2} \right| =:
        I_1(z,w) + I_2(z,w),
    \end{align*}
    Since $a(\zeta,z)$ depends smoothly on $z$, we may assume that
    $|a'(\zeta,z) - a'(\zeta,w)| \lesssim
    |z-w|$, and since the remaining integrand
    in $I_1(z,w)$ is integrable in $\zeta$ by Lemma~\ref{lem:estimate3}, $I_1(z,w)
\lesssim |z-w|$. The integrand in $I_2(z,w)$ is bounded
    by a constant times
    \begin{equation*}
        \left|\frac{\overline{\zeta_i - z_i}}{|\zeta-z|^4} - \frac{\overline{\zeta_i - w_i}}{|\zeta-w|^4}\right|,
    \end{equation*}
    and by the same argument as for the Bochner-Martinelli kernel on
    $\C^2$, see, e.g.,  \cite[Proposition~III.2.1]{LT},
    and using Lemma~\ref{lem:estimate3}, one obtains that $I_2(z,w)
    \lesssim |z-w|^\alpha$ for any $\alpha < 1$, and thus $\K_1$ is
    $C^\alpha$ for any $\alpha<1$.

    We next consider $k_2$. As above it is enough to consider one of
    the coefficients \linebreak $ b'(\zeta,z)\omega'_X(\zeta) \bar \eta_i
    \eta_j/|\eta|^4$ of one of the terms \eqref{k2}. In view of
    \eqref{eq:minrep} we can
    choose the coefficient $\omega'_X$ of $\omega_X$ in $L^q$ for $1\leq q<q(X)$. We divide the domain of integration $X$ into
    \begin{equation*}
        D_1 := X\cap B_{|z-w|/2}(z), \,\, D_2 := X\cap B_{|z-w|/2}(w),
        \,\,\text{ and } \,\, D_3 := X \setminus (D_1 \cup D_2),
    \end{equation*}
where $B_r(z)$ denotes a ball of radius $r$ centered at $z$.
    We choose  $2<q<q(X)$ and let $p = q/(q-1) < 2$ be the dual
    exponent. 
    Since $q(X) \leq 4$ by Corollary~\ref{qXbound}, $p > 4/3$.
    Using H\"older's inequality and Lemma~\ref{lem:estimate8}, we get
    \begin{equation*}
        \int_{\zeta\in D_\nu} |k_2'(\zeta,z)| \lesssim \left(\int_{\zeta\in D_\nu} \frac{1}{|\zeta-z|^{2p}}\right)^{1/p}
         \lesssim (|z-w|^{4-2p})^{1/p} = |z-w|^{4/p-2} 
    \end{equation*}
for $\nu=1,2$. By the same argument
    \begin{equation*}
        \int_{\zeta\in D_\nu} |k_2'(\zeta,w)| \lesssim |z-w|^{4/p-2}. 
    \end{equation*}
    For the integral on $D_3$, we use the following inequality,
    \begin{equation*}
        \left|\frac{\overline{\zeta_i - z_i}}{|\zeta-z|^4} - \frac{\overline{\zeta_i - w_i}}{|\zeta-w|^4}\right| \lesssim
        |z-w| \max\left\{\frac{1}{|\zeta-z|^4}, \frac{1}{|\zeta-w|^4} \right\},
    \end{equation*}
    see the proof of \cite[Lemma~III.2.2]{LT}. It follows that
    \begin{equation} \label{eq:star0}
        \left|\frac{\overline{(\zeta_i - z_i)}(\zeta_j - z_j)}{|\zeta-z|^4} - \frac{\overline{(\zeta_i - w_i)}(\zeta_j - w_j)}{|\zeta-w|^4} \right| \lesssim
        |z-w| \max\left\{\frac{1}{|\zeta-z|^3}, \frac{1}{|\zeta-w|^3} \right\},
    \end{equation}
    e.g., by assuming that $|\zeta-z|\leq |\zeta-w|$ and adding
and subtracting $(\overline{\zeta_i-w_i})(\zeta_j-z_j)/|\zeta-w|^{4}$ inside the absolute value sign on the left-hand side.
    Using H\"older's inequality as above, we get
    \begin{equation*}
        \int_{\zeta\in D_3} |k_2'(\zeta,z) - k_2'(\zeta,w)| \lesssim \left(\int_{\zeta\in D_3} \left|\frac{\overline{(\zeta_i - z_i)}(\zeta_j - z_j)}{|\zeta-z|^4} - \frac{\overline{(\zeta_i - w_i)}(\zeta_j - w_j)}{|\zeta-w|^4} \right|^p \right)^{1/p}.
    \end{equation*}
    By \eqref{eq:star0}, this is bounded by
    \begin{equation*}
        |z-w| \left(\int_{\zeta\in D_3}
          \max\left\{\frac{1}{|\zeta-z|^{3p}},
            \frac{1}{|\zeta-w|^{3p}} \right\} \right)^{1/p}.
    \end{equation*}
    Since $p > 4/3$, it follows from Lemma~\ref{lem:estimate2} that this is bounded by a constant times
\begin{equation*}
  |z-w| \big (|z-w|^{4-3p}\big )^{1/p} =
  |z-w|^{4/p-2}.
\end{equation*}

Since $p > 4/3$, we get that $4/p-2 < 1$.
Thus, it follows that $\mathcal{K}_2$ is $C^\alpha$ for any $\alpha <4/p-2$. We conclude that $\mathcal{K}$ is $C^\alpha$ for any $\alpha<4/p(X)-2$.

    \smallskip
    Since \eqref{eq:hoelder-estimate} holds uniformly for $z,w\in D$,
    if $\{\varphi_j\}_j$ and thus $\{\varphi_j'\}_j$ are bounded sequences in $L^\infty(X)$, then $\{(\mathcal{K} \varphi_j)'\}_j$
    are equicontinuous in the $C^\alpha(\overline{D})$-norm and thus $\mathcal{K}$ is compact by the Arzel\`a-Ascoli theorem.
\end{proof}



\begin{proof}[Proof of Corollary~\ref{cor:asheaves}]
The stalk of $\A_X$ at the singular point is a finite sum of currents of the form \begin{equation*}
        \xi_{\nu+1}\wedge (\mathcal{K}_\nu(\ldots \xi_3 \wedge \mathcal{K}_2(\xi_2\wedge \mathcal{K}_1\xi_1))),
    \end{equation*}
    where each $\mathcal{K}_i$ is an integral operator as in Theorem ~\ref{thm:main1},
    mapping forms on $D_i := \Omega_i \cap X$ to forms on $D_{i+1}$,
    where $\Omega = \Omega_{\nu+1} \subset\subset \Omega_\nu \subset\subset \dots \subset\subset \Omega_1 \subset\subset \C^3$
    are pseudoconvex domains, and $\xi_i$ are smooth forms on $D_i$.
The corollary now follows from Theorem~\ref{thm:main1}.
\end{proof}

\subsection{The operators $\hat{\K}$ and $\hat{\Pj}$ on forms with compact support} \label{ssec:KPcompact}

 Let $H\subset X$ be a compact Stein subset such that
 $D$ is relatively compact in the interior of $H$.
In \cite{AS1} are constructed integral operators, that we here denote by $\hat\K$ and $\hat\Pj$,
which map smooth forms with compact support in $D$ to smooth forms in $X\setminus\{0\}$
that vanish outside $H$, such that
\begin{equation}\label{dansa}
     \varphi(z) = \dq \hat\K \varphi(z) + \hat\K \dq \varphi(z) \
\text{ if } \rr =0,1,   \quad         \varphi(z) = \hat\Pj
\varphi(z) + \hat\K \dq \varphi(z)  \ \text{ if } \rr=2.
\end{equation}
In fact, $\hat\Pj$ maps forms with support in $D$ to smooth forms. Moreover,
$\hat\Pj\varphi=0$ unless $\rr=2$.
The kernels for these operators are obtained
 by choosing the weight $g$ differently; with notation as in Example~\ref{viktex},
we let $\chi=\chi(z)$ and we interchange the roles of $\zeta$ and $z$ in the functions $s_i(\zeta,z)$.
The resulting weight is then holomorphic in $\zeta$ and has compact support $H$  in $z$.

Since the proof of the mapping properties above
essentially only uses that
$g$ is smooth, it follows that an analogue of Theorem~\ref{thm:main1} holds also for these
operators.
The subscript $c$ denotes forms with compact support.

\begin{thm}\label{kompakt1}
In the situation of Theorem \ref{thm:main1},
the integral operator $\hat \K$ extends to an operator
$$
L^p_{0,\rr;c}(D)\to L^p_{0,\rr-1;c}(X),\ \   p(X) < p \leq \infty,  \ \ L^\infty_{0,\rr;c}(D)  \to C^\alpha_{0,\rr-1;c}(X), \ \ 0\leq \alpha<4/p(X)-2,
$$
and $\hat\Pj$ extends to an operator
$
L^p_{0,2;c}(D)\to C^{\infty}_{0,2;c}(X),\ \   p(X) < p \leq \infty.
$
\end{thm}

Note that the operators in fact map to forms with support in the fixed compact set $H$.

\section{Homotopy formulas}\label{sec:homotopy-formulas}
\begin{proof}[Proof of Theorem~\ref{thm:main2}]
By  \cite[Theorem~1.1]{AS1} the homotopy formula \eqref{koppla} holds
pointwise on $D_{reg}$ if $\varphi$ is smooth on $X$.
For $\varphi \in \Dom \dq_s^{(p)}$, let $\{\varphi_j\}_j$ be a sequence as in
Lemma~\ref{lem:cut-off2}.
We can assume that the $\varphi_j$ are smooth and bounded and with support away from the singularity $\{0\}$
(see the proof of Lemma ~\ref{lem:cut-off2} in \cite{LR2}).  Then the homotopy formula
\begin{equation}\label{eq:homotopy2}
      \varphi_j = \dq \mathcal{K} \varphi_j + \mathcal{K} \dq \varphi_j
\end{equation}
holds on $D$. In fact, since $\varphi_j$ is supported away from $X_{sing}$ all the terms are smooth on $D$,
see \cite[Lemma~6.1]{AS}.
By Theorem~\ref{thm:main1} 
we have that $\mathcal{K} \varphi_j \to \mathcal{K}\varphi$,
and $\mathcal{K}\dq \varphi_j \to \mathcal{K} \dq \varphi$ in $L^p(D)$.
It only remains to show that $\mathcal{K}\varphi \in \Dom\dq_s^{(p)}$.
Taking the limit $j\to \infty$ in \eqref{eq:homotopy2}
implies, by Theorem ~\ref{thm:main1},  that $\mathcal{K}\varphi\in \Dom \dq_w^{(p)}$ and
$
\dq \mathcal{K} \varphi = \varphi - \mathcal{K} \dq \varphi$
on $D$. As the $\varphi_j$ are bounded, $\{ \mathcal{K} \varphi_j\}_j$ is
by Theorem ~\ref{thm:main1} a sequence of bounded forms
such that  $\mathcal{K} \varphi_j \to \mathcal{K} \varphi$ and $\dq \mathcal{K} \varphi_j \to \dq \mathcal{K}\varphi$ in $L^p(D)$.
Hence, $\mathcal{K}\varphi \in \Dom \dq_s^{(p)}$ by Lemma ~\ref{lem:cut-off2}.
\end{proof}

\subsection{Proof of Theorem~\ref{thm:main3}} \label{ssec:star}

We first remark that the condition $(*)$ from the introduction is indeed independent of the sequence of cut-off functions $\{ \chi_k\}_k$.
Indeed, if $\{\chi_k'\}_k$ is another such sequence,
then $\chi_k-\chi_k'$ has compact support contained in $X^*$, and on this set, $\omega$ is $\dbar$-closed.
Thus,
\begin{equation*}
    \int_X \omega_X \wedge \dbar \chi_k \wedge \varphi - \int_X \omega_X \wedge \dbar \chi_k' \wedge \varphi
    = \int_X \omega_X \wedge (\chi_k - \chi_k') \wedge \dbar \varphi,
\end{equation*}
and this tends to $0$ as $k \to \infty$ by dominated convergence since $\omega_X \wedge \dbar \varphi$ is in $L^1(X)$.

\begin{proof}[Proof of Theorem \ref{thm:main3}]
We first note that it is enough to prove that
\begin{equation} \label{hoppla2}
    \varphi = \dbar \K \varphi + \K \dbar \varphi
\end{equation}
holds in the sense of distributions. Indeed, if it holds, then $\dbar \mathcal K \varphi=\mathcal K \dbar \varphi-\varphi$
is in $L^p(D)$ by Theorem \ref{thm:main1} since $\dbar \varphi\in L^p(X)$, and therefore
$\mathcal K\varphi\in \Dom\dq_w^{(p)}$.
We note that $p > p(X) \geq 4/3$ since $q(X) \leq 4$ by Corollary~\ref{qXbound}.

Let $\mu_k$ be the cut-off functions in Section ~\ref{ssec:cut-off}
and let $\varphi_k=\mu_k\varphi$. By the proof of Lemma~\ref{lem:cut-off} in \cite{LR2},
$\varphi_k \to \varphi$ in $L^p(X)$, $\mu_k \dbar \varphi \to \dbar \varphi$ in $L^p(X)$,
and $\dbar \mu_k \wedge \varphi \to 0$ in $L^{\lambda}(X)$ for $\lambda = 4p/(p+4) > 1$ since $p > 4/3$.
Since $\varphi_k$ has support away from the singularity, it follows as in the proof of Theorem~\ref{thm:main2}
that the homotopy formula \eqref{eq:homotopy2} holds on $D$.
Since $p>p(x)$, $\mathcal K\varphi_k$ converges to $\mathcal K\varphi$ in $L^p(D)$ by Theorem~\ref{thm:main1},
and it follows that $\dbar\mathcal K\varphi_k$ converges weakly to $\dbar \mathcal K\varphi$.
Since $\varphi_k = \mu_k \varphi \to \varphi$ and $\mu_k \dbar \varphi \to \dbar \varphi$ in $L^p(X)$, it follows
from Theorem ~\ref{thm:main1} that $\mathcal{K}(\varphi_k) \to \mathcal{K} \varphi$ and $\mathcal{K}(\mu_k \dbar \varphi) \to \mathcal{K} \dbar \varphi$ in $L^p(D)$.
Thus, using that $\dbar \varphi_k = \dbar \mu_k \wedge \varphi + \mu_k \dbar \varphi$, it follows that \eqref{hoppla2} holds if and only if
\begin{equation}\label{sjunga}
    \mathcal{K} ( \dbar \mu_k \wedge \varphi ) \to 0
\end{equation}
in the sense of distributions. If $\varphi$ is a $(0,\rr)$-form, then there is nothing to prove for $\rr=2$,
so let us assume that $\rr=1$.

We first consider the case when $p > \hat p(X)$. Then $\lambda > p(X)$
so \eqref{sjunga} holds by Theorem~\ref{thm:main1}, since $\dbar \mu_k \wedge \varphi \to 0$
in $L^\lambda(X)$.

It remains to prove that \eqref{sjunga} holds for $p(X) < p \leq \hat p(X)$
when $\varphi$ satisfies $(*)$.
To prove \eqref{sjunga} we decompose $\mathcal{K} = \mathcal{K}_1 + \mathcal{K}_2$
as in Proposition ~\ref{prop:Kestimate}.
We saw in the proof of Theorem ~\ref{thm:main1} that $\mathcal K_1$ is
a bounded linear operator $L^p_{0,\rr+1}(X)\to L^p_{0,\rr}(D)$ for any $1\leq p\leq
\infty$. It follows, in particular, that $\mathcal{K}_1 ( \dbar \mu_k \wedge \varphi ) \to 0$
in the sense of distributions, since
$\dbar \mu_k \wedge \varphi \to 0$ in $L^{\lambda}_{0,\rr+1}(X)$ where $\lambda > 1$ since $p > 4/3$.

    We next consider $\mathcal{K}_2$. In view of Proposition
    ~\ref{prop:Kestimate} we may assume that the kernel is of the
    form \eqref{k2}.
To prove \eqref{sjunga} for $\mathcal K_2$ we need to prove that
    \begin{equation}\label{eq:fubini0}
        \langle \mathcal{K}_2( \dbar\mu_k \wedge \varphi) , \xi\rangle
        = \int_z \xi(z) \wedge \int_\zeta b(\zeta,z) \wedge
        \omega_X(\zeta) \wedge \frac{\overline{\eta_i}\eta_j}{|\eta|^4} \dbar \mu_k(\zeta) \wedge\varphi(\zeta)
    \end{equation}
tends to $0$ as $k\rightarrow\infty$ for each test form $\xi$.
    By Fubini's theorem, up to signs \eqref{eq:fubini0} is equal to
    \begin{equation} \label{eq:fubini1}
        \int_\zeta \omega_X \wedge \dbar \mu_k \wedge \varphi\wedge
        \int_z b(\zeta,z) \wedge \xi(z)
        \frac{\overline{\eta_i}\eta_j}{|\eta|^4}.
    \end{equation}
    We denote the inner integral with respect to $z$ by $\gamma(\zeta)$.
Note that
    \begin{equation*}
       \gamma = \int_z
       \frac{c(\zeta,z)(\overline{\zeta_i-z_i})(\zeta_j-z_j)}{|\zeta-z|^4},
    \end{equation*}
where $c(\zeta, z)$ is a smooth $(2,2)$-form.
Now $ |\gamma(\zeta) - \gamma(0)|$ is bounded by
 \begin{equation*}
        \int_z |c(\zeta,z) - c(0,z)| \left|\frac{(\overline{\zeta_i -
          z_i})(\zeta_j-z_j)}{|\zeta-z|^4}\right| +
      |c(0,z)|\int_z \left|\frac{(\overline{\zeta_i-z_i})(\zeta_j-z_j)}{|\zeta-z|^4}-
      \frac{\overline{z_i}z_j}{|z|^4}\right|:=I_1+I_2.
\end{equation*}
Since $c(\zeta, z)$ depends smoothly on $\zeta$, $|c(\zeta,z) -
c(0,z)|\lesssim |\zeta|$ and thus, in view of Lemma
~\ref{lem:estimate3},
$I_1 \lesssim |\zeta|$.
Moreover, by \eqref{eq:star0} and Lemma~\ref{lem:estimate3},
$I_2\lesssim |\zeta|\int_z\max(|\zeta-z|^{-3},
  |z|^{-3} )\lesssim |\zeta|$.

Since $\varphi$ satisfies $(*)$, and this condition is independent
of the choice of $\chi_k$, we may assume that $\chi_k = \mu_k$, and thus
$\int_\zeta \omega_X \wedge \dbar \mu_k \wedge \varphi \wedge \gamma(0)$ tends to $0$ as $k \to \infty$.
It follows from \eqref{eq:cutoff2} that $|\dbar\mu_k\wedge (\gamma(\zeta)-\gamma(0))|\leq C \chi_k(|\zeta|)$ when
$|\zeta|\ll 1$ and where $\chi_k$ is as in Section
~\ref{ssec:cut-off}. Since $p>p(X)$, by H\"older's inequality,
$\omega_X\wedge \varphi\in L^1(X)$ and therefore
$\lim_k \int_\zeta \omega_X \wedge \dbar \mu_k \wedge \varphi \wedge
(\gamma(\zeta)-\gamma(0))=0$ by dominated convergence. Hence \eqref{eq:fubini1} tends
to $0$ as $k\to \infty$.
\end{proof}

It follows from the proof of Theorem~\ref{thm:main3} that if $\varphi \in \Dom \dq_w^{(p)}$, with $p > \hat p(X)$, then
$(*)$ is automatically fulfilled for $\varphi$, since if $\mu_k$ is as in Section~\ref{ssec:cut-off}, then
\begin{equation*}
    \int_X \omega_X \wedge \dbar\mu_k \wedge \varphi \to 0
\end{equation*}
by H\"older's inequality.

\smallskip
It is worth noting that since the condition $(*)$ does not depend on $p$, we have the following
consequence of Theorem ~\ref{thm:main3}:

\begin{cor}
If the $\dq_w$-equation is locally solvable on a canonical surface for
some $p_0> p(X)$,
then is is locally solvable for all $p\geq p_0$.
\end{cor}

Morally this means that the number of obstructions to solving the
$\dq_w$-equation in the $L^p$-sense is decreasing in $p$. Theorem~1.1
in \cite{RMatZ} shows that the same kind of phenomenon holds for homogeneous varieties
with an isolated singularity.
%

\smallskip
Let $\varphi\in L^p_{0,1}(X)$, where $p(X)<p\leq 4$. Assume that
$\varphi\in\Dom \dbar_s$. Then by Theorem ~\ref{thm:main2}, $\varphi =
\dq_s \mathcal{K} \varphi$ which implies particularly that $\varphi = \dq_w \mathcal{K} \varphi$.
Hence, $\varphi$ satisfies $(*)$. It would be interesting to know whether the converse is also true,
i.e., if $\varphi$ satisfies $(*)$, does it follow that $\varphi\in\Dom\dq_s$?

\subsection{Proof of Theorem~\ref{functions}}\label{gorilla}
As explained after Proposition~\ref{prop:Kestimate}, the operator $\mathcal{P}$ is defined by an
integral kernel $P(\zeta,z)$ that is smooth with compact support in $\zeta$, and holomorphic in $z$.
Therefore  $\mathcal{P}$ extends to a compact operator $\mathcal{P} : L^p(X) \to \Ok(D)$, cf. the proof of Theorem~\ref{thm:main1}. 


The formula \eqref{trottel} for $\varphi \in \Dom \dbar_s^{(p)}$ and $p(X) < p \leq 4$ is proved in the same way
as  Theorem~\ref{thm:main2} above, using that \eqref{trottel} holds for the smooth functions $\varphi_j$, and
that $\mathcal{P} \varphi_j \to \mathcal{P} \varphi$.


Now assume that  $\varphi$ is a function in $\Dom\dq_w^{(p)}$, where $p(X) < p$.
We say that $\varphi$ satisfies $(*)$ if
\begin{equation} \label{star-functions}
    \int_X  \omega_X \wedge \dbar\chi_k \wedge \varphi \wedge \alpha \to 0
\end{equation}
for any smooth $\dbar$-closed $(0,1)$-form $\alpha$ and sequence $\chi_k$ as in Section~\ref{ssec:star}.
In particular, if $\varphi$ is $\dbar$-closed, i.e., holomorphic on the regular part of $X$, then 
as $X$ is a canonical surface, $X_{sing}$ has codimension $2$ and thus  $\varphi$ is bounded in a neighborhood
of the singularity at the origin. Therefore $\varphi\in L^p$ for any $p\geq 1$ and it follows as for $(0,1)$-forms above that
$(*)$ is satisfied.

If $\varphi \in \Dom\dq_w^{(p)}$ and $p > \hat p(X)$, then \eqref{trottel} can be verified in the
same way as Theorem~\ref{thm:main3} above.
If instead $p(X) < p \leq \hat p(X)$ and $\varphi$ satisfies \eqref{star-functions},
one just needs to make minor modifications. Namely, at the point where one considers $\gamma(\zeta)-\gamma(0)$,
then $\gamma$ is a $(0,1)$-form, and one then writes $\gamma = \sum \gamma_i d\bar{\zeta}_i$,
decomposes $\gamma_i(\zeta) = \gamma_i(0) + (\gamma_i(\zeta)-\gamma_i(0))$ and proceeds as in the proof above.
The condition \eqref{star-functions} is then finally applied with $\alpha = \gamma_i(0) d\bar{\zeta}_i$.

\subsection{Homotopy formulas with compact support}\label{kompakta}
We get versions of Theorems ~\ref{thm:main2} and ~\ref{thm:main3} 
also for the operators in Theorem \ref{kompakt1}.

\begin{thm} \label{kompakt2}
Assume we are in the situation of Theorem ~\ref{thm:main1}.

\smallskip
\noindent (i) Let $p(X)< p \leq 4$ and
let $\varphi$ be an $(0,\rr)$-form in $\Dom\dq_s^{(p)} \subset L^p_c(D)$.
Then $\hat\K\varphi \in \Dom\dq_s^{(p)}$,
\begin{equation}\label{kopplaav}
      \varphi(z) = \dq_s \hat\K \varphi(z) + \hat\K  \dq_s\varphi(z) \
\text{ if } \rr=0,1,   \quad
      \varphi(z) = \dq_s\hat\K  \varphi(z) + \hat\Pj \varphi(z)  \ \text{ if } \rr=2.
\end{equation}

\smallskip
\noindent (ii) If $\hat p(X) < p \leq \infty$ and $\varphi$ is a $(0,\rr)$-form with $\rr=0,1$, in $\Dom\dq_w^{(p)} \subset L^p_c(D)$,
then $\hat\K\varphi \in \Dom\dq_w^{(p)}$ and
\begin{equation}\label{hopplaav}
      \varphi(z) = \dq_w \hat\K  \varphi(z) + \hat\K \dq_w \varphi(z).
\end{equation}
If $p(X)< p\le \hat p(X)$, and in addition $\varphi$ satisfies the condition $(*)$, then the same
conclusion holds.

\smallskip
\noindent (iii)
If $p(X)< p \leq \infty$ and $\varphi$ is a $(0,2)$-form in $\Dom\dq_w^{(p)} \subset L^p_c(D)$, then
$\hat\K\varphi \in \Dom\dq_w^{(p)}$ and
\begin{equation}\label{hopplaav2}
      \varphi(z) = \dq_w\hat\K  \varphi(z) + \hat\Pj \varphi(z).
\end{equation}
The $(0,2)$-form $\varphi$ satisfies that $\hat\Pj\varphi=0$ if \eqref{kokett} holds, and if $p > \hat p(X)$, then the converse holds.
\end{thm}

These statements are proved essentially by the same arguments as in
the proofs of Theorems  ~\ref{thm:main2} and  ~\ref{thm:main3}.
For \eqref{kopplaav}, notice that
as $\varphi$ has compact support in $D$, when choosing the approximating sequences
$\{\varphi_j\}_j$  we may, in addition, assume that the $\varphi_j$ have compact support
in $D$ as well.

For the last statement, notice that the kernel for $\hat\Pj$ has the form $h\omega_X$
with respect to $\zeta$ in a Stein neighborhood of  the support of $\varphi$. Since $X$ is Stein, we can assume that
$h$ is holomorphic on $X$ and so $\hat\Pj\varphi=0$ if \eqref{kokett} holds.  Conversely, if
$\hat\Pj\varphi=0$, then $u=\hat\K\varphi$ is a solution to $\dbar u=\varphi$ with support on the compact set $H$,
see Section~\ref{ssec:KPcompact}.
It then follows that \eqref{kokett} holds if and only if $u$ satisfies $(*)$, which as we saw in Section~\ref{ssec:star}
is automatically satisfied for $p > \hat p(X)$.

Note that if $\varphi$ is a $(0,1)$-form in $L^p_c(X)$ and $\dbar \varphi = 0$, then it automatically satisfies $(*)$, so
if $p> p(X)$, then $u = \hat{\mathcal{K}}\varphi$ is a solution with compact support to
$\dbar u = \varphi$.

\subsection{On the domain of the $\dbar_X$-operator}\label{snabel}

The setting in \cite{AS} is rather different compared to this article. Here
we are mainly concerned
with forms on $X$ with coefficients in $L^p$, while in \cite{AS}, the type of forms/currents
considered, denoted $\mathcal{W}_X^\rr$, are ``generically'' smooth, see \cite{AW}. They
include principal value currents $\alpha/f$, where
$f$ is holomorphic and $\alpha$ is smooth, and direct images of such currents, but with no
growth restrictions on the singularities.
For the precise definition of the sheaf $\mathcal{W}_X^\rr$ we refer to \cite{AS,AW}.
The $\dbar$-operator considered in \cite{AS}
is somewhat different from $\dq_s$ and $\dq_w$ considered here.
For currents in $\mathcal{W}_X^\rr$, one can define the product with the structure form
$\omega_X$ associated to the variety. A current $\mu \in
\mathcal{W}_X^\rr$ lies in $\Dom \dq_X$ if $\dbar\mu\in \mathcal
W_X^{\rr+1}$ and $\dq (\mu \wedge \omega_X) =\dq \mu \wedge \omega_X$.
%
Combining our results about $\mathcal{K}$ and the $\dq_w$- and $\dq_s$-operator
with some properties about the $\mathcal{W}_X$-sheaves,
we obtain a result for the $\dq_X$-operator, providing a partial answer
to a question in \cite{AS}, cf.\ the paragraph at the end of page 288 in \cite{AS}.

\begin{thm}\label{thm:main5}
In the situation of Theorem~\ref{thm:main1},
let $p(X) < p \leq 4$ and $\varphi\in \Dom \dq^{(p)}_s \cap \mathcal{W}_X^\rr(X)$.
Then $\mathcal{K} \varphi \in  \Dom\dq_X.$
\end{thm}

Note that if $\varphi \in \Dom \dq^{(p)}_w$, where $p > \hat p(X)$, then by Lemma~\ref{lem:cut-off},
$\varphi \in \Dom \dq^{(\lambda)}_s$, where $\lambda > p(X)$, so also in this case, $\mathcal{K}\varphi \in \Dom \dq_X$.

\begin{proof}
First, by \cite[Proposition~1.5]{AS}, $\mathcal K\varphi\in \mathcal
W(D)$. Then in particular $\dbar \mathcal K\varphi$ is a
pseudomeromorphic current, see, e.g., \cite{AW}. Moreover, by Theorem
\ref{thm:main2} $\mathcal K\varphi\in\Dom\dbar_s^{(p)}$, and thus
$\dbar \mathcal K\varphi \in L^p(D)$. Now a pseudomeromorphic current
in $L^p$ is in fact in $\mathcal W$, cf.\ \cite{AW}. Hence
$\dbar \mathcal K\varphi \in \mathcal W(D)$.

Since $\mathcal K\varphi \in \Dom \dbar_s^{(p)}$, there is a sequence
of smooth forms $\psi_j$ with support away from the singularity such
that $\psi_j\to \mathcal K\varphi$ and $\dbar\psi_j\to \dbar \mathcal
K\varphi$ in $L^p(D)$.
Since $\omega_X\in L^q(X)$ for each $q<q(X)$, by H\"older's
inequality,
$\psi_j\wedge \omega_X$ and $\dbar \psi_j\wedge \omega_X$ converge in
$L^1(D)$ to
$\mathcal K\varphi\wedge \omega_X$ and $\dbar\mathcal K\varphi \wedge
\omega$, respectively. Hence
\begin{equation*}
\dbar (\mathcal K\varphi \wedge \omega_X) =
\lim_j \dbar(\psi_j\wedge \omega_X) =
\lim_j \dbar\psi_j\wedge \omega_X =
\dbar \mathcal K\varphi \wedge \omega_X.
\end{equation*}
We conclude that $\mathcal K\varphi\in \Dom \dbar_X$.
\end{proof}

\section{Examples and counterexamples}
\label{sec:examples}

In this section, we study the condition $(*)$ and $\dq_w$-Koppelman formulas for all types of canonical surface singularities:
$A_n$, $n\geq 1$, $D_n$, $n\geq 4$, $E_6$, $E_7$ and $E_8$.
We focus on the important case of $L^2$-cohomology, i.e., $p=2$.
However, we also get some statements for $p\geq 2$.
All in all we obtain a complete picture about the solvability of the $\dq_w$-equation in the $L^2$-sense
at canonical surface singularities.

\subsection{The $A_n$-singularities}
\label{ssec:An}

Recall that the $A_n$-singularity for $n\geq 1$ is the variety $X = \{
f(\zeta) = 0 \} \subset \C^3$, where $f(\zeta) = \zeta_1 \zeta_2 - \zeta_3^{n+1}$.

\begin{thm}\label{thm:An}
Let $X$ be (a neighborhood of the origin of) the $A_n$ singularity
$\{\zeta_1\zeta_2=\zeta_3^{n+1}\}\subset \C^3$,
let $p\geq 2$, and let $\varphi \in\Dom \dq_w^{(p)} \subset
L^p_{0,\rr}(X)$. Then $\varphi$ satisfies the condition $(*)$.
\end{thm}

In combination with Theorem ~\ref{thmA} we get the following.

\begin{cor}
The $\dq_w$-equation is solvable at the $A_n$-singularity in the $L^p$-sense for $p\geq 2$.
\end{cor}

\begin{proof}[Proof of Theorem \ref{thm:An}]
Note that $X$ has a branched $n+1$ to $1$ covering $\C^2 \rightarrow X$ given by $\pi(s,t) = (s^{n+1},t^{n+1},st)$.
If $\beta=\frac{i}{2}\sum d\zeta_j\wedge d\bar\zeta_j$ is the standard K\"ahler form on $\C^3$, so that $\beta^2|_X = dV_X$, then we obtain:
\begin{equation}\label{eq:omega-pullback}
    \pi^*\beta^2 = 2 (n+1)^2\big (|s|^{2n+2} + |t|^{2n+2} +
    (n+1)^2 |st|^{2n}\big ) dV(s,t).
\end{equation}
Recall that by \eqref{eq:structure-form2}, we can choose a representation
$\omega_X = \sum \omega_{i,j} ~d\zeta_i \wedge d\zeta_j$  of the structure form
such that $|\omega_{i,j}| \lesssim 1/|\partial f|$.
Since $\pi^*|\partial f|^2=|s|^{2n+2} + |t|^{2n+2} + (n+1)^2
|st|^{2n}$
we get from \eqref{eq:omega-pullback} that
\begin{equation}\label{eq:estimate01}
\pi^* |\omega_{i,j}|^2 \pi^* \beta^2 \lesssim dV(s,t).
\end{equation}

Let $\mu_k$ be cut-off functions as in Section ~\ref{ssec:cut-off} and
let $D_k = \{\zeta\in X: e^{-e^{k+1}} < |\zeta| < e^{-e^k} \}$. Then,
in view of \eqref{eq:cutoff2},
the integral in $(*)$ is a finite sum of integrals, which are
bounded by constants times integrals of the form
\begin{equation*}
\int_{D_k} |\omega_{i,j}| \frac{1}{|\zeta|\big |\log
      |\zeta|\big |} |\gamma| \beta^2\leq
\left (\int_{D_k} \frac{|\omega_{i,j}|^2 \beta^2}{|\zeta|^2
    \log^2|\zeta|} \right )^{1/2}
\left (\int_{D_k} |\gamma|^2 \beta^2 \right )^{1/2}
=: (I_{1,k})^{1/2} (I_{2,k})^{1/2}
\end{equation*}
where $\gamma \in L^2_{0,0}(X)$ and the inequality follows from the Cauchy-Schwarz inequality.
Note that $I_{2,k} \rightarrow 0$ for $k\rightarrow \infty$ by dominated convergence because $\gamma \in L^2_{0,0}(X)$
and the domain of integration shrinks to $0$. Therefore it is enough to show
that $I_{1,k}$ is uniformly bounded in ~$k$.

From \eqref{eq:estimate01} it follows that
\begin{equation}\label{eq:final}
I_{1,k} \lesssim \int_{\pi^{-1}(D_k)} \frac{dV(s,t)}{\big( |s|^{2n+2} + |t|^{2n+2} + |st|^2\big) \log^2\big( |s|^{2n+2} + |t|^{2n+2} + |st|^2\big)}.
\end{equation}
Let us decompose
\begin{equation*}
    \pi^{-1}(D_k) = \{ e^{-2e^{k+1}} \leq |s|^{2n+2} +
    |t|^{2n+2} + |st|^2 \leq e^{-2e^k} \}\subset \C^2
\end{equation*}
into $E_k := \pi^{-1}(D_k) \cap \{(1/2) e^{-2e^{k+1}} \leq |st|^2
  \leq e^{-2e^k} \}$ and $F_k:=\pi^{-1}(D_k) \setminus E_k.$
Note that $E_k\subset  E_k' := \{ (s,t) \mid e^{-e^{k+2}} \leq |st|
\leq e^{-e^k} \}$ because $2^{-1/2} e^{-e^{k+1}} >
e^{-e^{k+2}}$. Therefore
we obtain by \cite[Appendix B]{RDuke}:
\begin{equation*}
   \int_{E_k} \frac{dV(s,t)}{\big( |s|^{2n+2} + |t|^{2n+2} + |st|^2\big) \log^2\big( |s|^{2n+2} + |t|^{2n+2} + |st|^2\big)}
    \leq \int_{E_k'} \frac{dV(s,t)}{|st|^2 \log^2\big( |st|^2\big)} \leq C
\end{equation*}
uniformly in $k$.

Next note that on $F_k$, $|s|^{2n+2} + |t|^{2n+2} \geq |st|^2$.
Therefore if $(s,t)\in F_k$ satisfies $|s| \leq |t|$, then $|st|^2 \leq
2|t|^{2n+2}$, and so
\begin{eqnarray*}
e^{-2e^{k+1}} \leq &|s|^{2n+2} + |t|^{2n+2} + |st|^2& \leq 4 |t|^{2n+2},\\
|t|^{2n+2} \leq & |s|^{2n+2} + |t|^{2n+2} + |st|^2 &  \leq
                                                       e^{-2e^k},
\end{eqnarray*}
and $|s|^2 \leq 2|t|^{2n}$.
By symmetry we get that
\begin{equation*}
   F_k \subset  \{A_k\leq |s| \leq B_k, 0 \leq |t| \leq \sqrt 2|s|^n \}
    \cup \{A_k \leq |t| \leq B_k, 0 \leq |s| \leq \sqrt 2|t|^n \},
\end{equation*}
where $A_k =e^{-{e^{k+1}/(n+1)}}/2^{1/(n+1)}$ and
$B_k=e^{-e^{k}/(n+1)}$.
Now by integration in polar coordinates 
\begin{align*}
    &\int_{F_k} \frac{dV(s,t)}{\big( |s|^{2n+2} + |t|^{2n+2} + |st|^2\big) \log^2\big( |s|^{2n+2} + |t|^{2n+2} + |st|^2\big)} \\
    \lesssim  &\int_{A_k}^{B_k} \int_0^{\sqrt 2r_2^n} \frac{r_1 r_2 dr_1 dr_2}{r_2^{2n+2} \log^2(r_2)}
    = \int_{A_k}^{B_k}\frac{dr_2}{r_2 \log^2(r_2)} \to 0
\end{align*}
when $k \to \infty$ because the integrand is integrable over, say $[0,1/2]$.
Thus \eqref{eq:final} is uniformly bounded in $k$.
\end{proof}


\subsection{On the Euler characteristics of the structure sheaf} 

As a preparation for the proof of the existence of obstructions for solvability of the $\dq_w$-equation
at canonical singularities in the $L^2$-sense,
we need some observations on the behaviour of the Euler
characteristics of the structure sheaf
under resolution of singularities.

Let $\mathcal{F} \rightarrow X$ be a coherent analytic sheaf over a
compact complex space $X$ of pure dimension $n$, and let $\chi
(\mathcal F)$ be the Euler characteristic of $\mathcal F$,
\begin{equation*}
\chi\big(\mathcal{F}\big) := \sum_{j=0}^n (-1)^j \dim H^j\big(X,\mathcal{F}\big).
\end{equation*}
If $D$ is a divisor on $X$, associated to a line bundle $L \rightarrow X$,
then $\chi\big(\OO_X(D)\big)$ is the holomorphic Euler characteristic of $L$.

\begin{prop}\label{prop:examples}
Let  $\pi: M \rightarrow X$ be a resolution of singularities of a
compact surface $X$ with at most canonical singularities. Then
$
\chi(\OO_X) = \chi(\OO_M)$.
\end{prop}

\begin{proof}
Since $X$ is a normal space, $\pi_* \OO_M=\OO_X$.
Moreover, canonical singularities are rational so that
$R^k \pi_* \OO_M = 0$ for $ k>0$.
Hence, the Leray spectral sequence gives $H^k(X,\OO_X)  \cong
H^k(M,\OO_M)$ for $k\geq 0$.
\end{proof}

If we assume that the $\dq_w$-equation is locally solvable in the $L^2$-sense, then we obtain another representation
of $\chi(\OO_X)$ for arbitrary normal complex surfaces.

\begin{thm}\label{thm:examples}
Let $X$ be a compact normal complex surface, $\pi: M\rightarrow X$ a resolution of singularities
with only normal crossings, $Z:=\pi^{-1}(X_{sing})$ the unreduced exceptional divisor
and $E:=|Z|$ the exceptional divisor.
If the $\dq_w$-equation is locally solvable in the $L^2$-sense
for $(0,1)$-forms, then
\begin{equation}\label{eq:chi2}
\chi\big(\OO_X\big) =\chi\big( \OO_M(Z-E) \big).
\end{equation}
\end{thm}

\begin{proof}
Following \cite[Section~2.1]{RDuke}, let $\mathcal C_{0,\rr}$ denote the
fine sheaves $L_{0,\rr}^{2, loc}\cap \Dom \dbar_w$ and consider the sheaf complex
\begin{equation}\label{eq:examples}
0 \rightarrow \OO_X \longrightarrow \mathcal{C}_{0,0} \overset{\dq_w}{\longrightarrow} \mathcal{C}_{0,1} \overset{\dq_w}{\longrightarrow}
\mathcal{C}_{0,2} \rightarrow 0.
\end{equation}
It is easy to see that \eqref{eq:examples} is exact at $\mathcal{C}_{0,0}$
because $X$ is normal; a germ $f \in \ker\dq_w \subset \mathcal{C}_{0,0}$ is a holomorphic
function on the regular locus of $X$, and so it is also strongly holomorphic by normality.
Moreover, \eqref{eq:examples}
is exact at $\mathcal{C}_{0,2}$, see \cite[Theorem~4.3]{OR}.
(It is usually not difficult to solve $\dq$-equations in the highest degree, see also \cite{S}.)
In general, \eqref{eq:examples} is not necessarily exact at $\mathcal{C}_{0,1}$,
but here, we assume that this is the case.
Thus \eqref{eq:examples} is a fine resolution of $\OO_X$;  in
particular $H^k(X,\OO_X)= H^k(\Gamma(X,\mathcal{C}_{0,\bullet}))$.
By \cite[Theorem~1.13]{RDuke}
$H^k(\Gamma(X,\mathcal{C}_{0,\bullet}))= H^k( M, \OO_M(Z-E))$, and so
\[
H^k(X,\OO_X)=H^k( M, \OO_M(Z-E)),
\]
which proves \eqref{eq:chi2}.
\end{proof}

Combining Proposition ~\ref{prop:examples}
and Theorem ~\ref{thm:examples} we get: 

\begin{cor}\label{cor:examples}
Let $X$ be a compact complex surface with at most canonical singularities. If the $\dq_w$-equation is locally solvable
in the $L^2$-sense for $(0,1)$-forms on $X$, then
\begin{equation}\label{eq:agenus}
\chi \big(\OO_M) = \chi\big( \OO_M(Z-E)\big)
\end{equation}
for any resolution of singularities $\pi: M\rightarrow X$ with only normal crossings.
\end{cor}

So, if we are looking for obstructions to solvability of the $\dq_w$-equation in the $L^2$-sense
at canonical singularities, we just need to find configurations violating \eqref{eq:agenus}.

\subsection{Obstructions for $\dq_w$ at canonical singularities}
\label{ssec:obstructions}

\begin{thm}\label{thm:examples2}
There exist obstructions to local solvability of the $\dq_w$-equation in the $L^2$-sense for $(0,1)$-forms at singularities of type
$D_n$, $n\geq 4$, $E_6$, $E_7$ and $E_8$.
\end{thm}

Hence, $(*)$ does not hold for all $\varphi \in \ker\dq_w \subset L^2_{0,1}$ at such singularities.

\begin{proof}
Let $X$ be a projective variety with a single singularity of one of the types above,
and $\pi: M \rightarrow X$ a resolution of singularities with only
normal crossings.
In view of the discussion above it suffices to show that
\eqref{eq:agenus} does not hold.
For the $D_n$-singularities, $n\geq4$, this was proved in the proof of
Theorem~4.8 in \cite{P} using the Riemann-Roch formula for regular complex surfaces
\begin{equation*}
\chi\big(\OO_M(Z-E)\big) = \chi\big( \OO_M\big) + \frac{1}{2} \big(
(Z-E)\cdot(Z-E) - (Z-E)\cdot K\big),
\end{equation*}
where $K$ is the canonical divisor on $M$.
Since $\OO(K)$ is trivial on a neighborhood of the exceptional set, $Z_j\cdot K=0$
for any irreducible component $Z_j$ of the exceptional set, cf.\
\cite[page~135]{Du}, and thus $(Z-E)\cdot K =0$. Pardon proved that
$(Z-E)\cdot (Z-E)=-2$
so that
\begin{equation}\label{eq:formula-chi}
\chi \big(\OO_M) = \chi\big( \OO_M(Z-E)\big) + 1
\end{equation}
and in particular \eqref{eq:agenus} does not hold.

For the remaining singularities, $E_6$, $E_7$ and $E_8$, we proceed analogously to \cite{P}
and show 
that \eqref{eq:formula-chi} holds also for these singularities.
Now let $\pi: M \rightarrow X$ be the minimal resolution of $X$. Then the
exceptional divisor $Z$ has normal crossings and the irreducible
components $E_j$ have self-intersection $-2$ and pairwise
intersections according to the Dynkin diagrams of $E_6$, $E_7$ or
$E_8$, see, e.g., \cite{Du}.
The labels of the nodes in the following diagrams are the multiplicities of the
corresponding divisors in the unreduced fundamental cycle $Z$, cf., e.g.,
\cite[Example~7.2.5]{Ishii} and \cite[Proposition~3.8]{BPV}.

\begin{center}
\begin{tikzpicture}
    \filldraw (0,4) circle [radius=2pt] node[above] {$1$};
    \filldraw (1,4) circle [radius=2pt] node[above] {$2$}; 
    \filldraw (2,4) circle [radius=2pt] node[above] {$3$}; 
    \filldraw (3,4) circle [radius=2pt] node[above] {$2$};
    \filldraw (4,4) circle [radius=2pt] node[above] {$1$};
    \filldraw (2,3) circle [radius=2pt] node[left] {$2$};
    \draw (0,4) -- (1,4);
    \draw (1,4) -- (2,4);
    \draw (2,4) -- (3,4);
    \draw (3,4) -- (4,4);
    \draw (2,4) -- (2,3);

    \draw (2,5) node {$E_6$};

    \filldraw ( 6,4) circle [radius=2pt] node[above] {$2$};
    \filldraw ( 7,4) circle [radius=2pt] node[above] {$3$};
    \filldraw ( 8,4) circle [radius=2pt] node[above] {$4$};
    \filldraw ( 9,4) circle [radius=2pt] node[above] {$3$};
    \filldraw (10,4) circle [radius=2pt] node[above] {$2$};
    \filldraw (11,4) circle [radius=2pt] node[above] {$1$};
    \filldraw (8,3) circle [radius=2pt] node[left] {$2$};
    \draw ( 6,4) -- ( 7,4);
    \draw ( 7,4) -- ( 8,4);
    \draw ( 8,4) -- ( 9,4);
    \draw ( 9,4) -- (10,4);
    \draw (10,4) -- (11,4);
    \draw (8,4) -- (8,3);

    \draw (8.5,5) node {$E_7$};

    \filldraw (2.5,1) circle [radius=2pt] node[above] {$2$};
    \filldraw (3.5,1) circle [radius=2pt] node[above] {$4$};
    \filldraw (4.5,1) circle [radius=2pt] node[above] {$6$};
    \filldraw (5.5,1) circle [radius=2pt] node[above] {$5$};
    \filldraw (6.5,1) circle [radius=2pt] node[above] {$4$};
    \filldraw (7.5,1) circle [radius=2pt] node[above] {$3$};
    \filldraw (8.5,1) circle [radius=2pt] node[above] {$2$};
    \filldraw (4.5,0) circle [radius=2pt] node[left] {$3$};
    \draw (2.5,1) -- (3.5,1);
    \draw (3.5,1) -- (4.5,1);
    \draw (4.5,1) -- (5.5,1);
    \draw (5.5,1) -- (6.5,1);
    \draw (6.5,1) -- (7.5,1);
    \draw (7.5,1) -- (8.5,1);
    \draw (4.5,1) -- (4.5,0);

    \draw (5.5,2) node {$E_8$};
\end{tikzpicture}
\end{center}
%
This means that in the case of the $E_6$-singularity,
we can label the irreducible components of $Z$ so that $Z-E=2Z_0 + Z_1 +
Z_2 + Z_3$ and $Z_\nu^2=-2$, $Z_0 \cdot Z_\mu=1$ if $\mu\geq 1$,
and $Z_\nu\cdot Z_\mu=0$ if $\mu>\nu\geq 1$.
For the $E_7$-singularity we can label the irreducible components of
$Z$ so that $Z-E=3Z_0 + 2Z_1 + Z_2 + Z_3+ 2Z_4+Z_5$ and
$Z_\nu^2=-2$ for all $\nu$, $Z_0 \cdot Z_1 = Z_0 \cdot Z_3 = Z_0 \cdot
Z_4=Z_1 \cdot Z_2=Z_4 \cdot Z_5=1$, and $Z_\nu \cdot Z_\mu=0$ for all
other combinations of $\nu\neq \mu$.
Finally, for the $E_8$-singularity, we have $Z-E=
5Z_0 + 3Z_1 + Z_2 + 2Z_3+ 4Z_4+3Z_5 + 2Z_6 + Z_7$
and $Z_\nu^2=-2$ for all $\nu$, $Z_0 \cdot Z_1 = Z_0 \cdot Z_3 = Z_0
\cdot Z_4=Z_1 \cdot Z_2=Z_4 \cdot Z_5= Z_5 \cdot Z_6 = Z_6\cdot
Z_7=1$, and $Z_\nu \cdot Z_\mu=0$ for all other combinations of
$\nu\neq \mu$.
In all three cases a computation yields that $(Z-E)\cdot (Z-E)=-2$,
which implies \eqref{eq:formula-chi}.
\end{proof}

\section{Appendix -- Integral estimates on analytic varieties} \label{sect:basic-estimates}

In this section we recall for convenience of the reader some basic integral estimates
for analytic varieties from \cite{LR2}.
Let $i\colon X\to \Omega'\subset \C^N$ be an analytic variety of pure dimension $n$.
We consider $X$ as a Hermitian complex space with the restriction of the standard metric from $\C^N$,
i.e., $X^*:=X_{reg}$ of $X$ carries the induced Hermitian metric.
With respect to the volume element induced by this metric,  $X_{sing}$ is a null set,
and we denote by $dV_X$ the extension to $X$ of the volume element on $X^*$.
Let $B_r(z)$ be the ball of radius $r>0$ centered at the point $z\in\C^N$.
The results below are all consequences of Lemma~2.1 in \cite{LR2} which asserts
that radial integrals on $X$ behave like in $\C^n$, which in turn
follows from the fact that the volume of a ball $X\cap B_r(z)$ is $\sim
r^{2n}$, cf.\ \cite[Consequence~III.5.8]{De}.


\begin{lem}[\cite{LR2}, Lemma~2.2]\label{lem:estimate2}
Let $X\subset \C^N$ be an analytic variety of pure dimension $n$,
$K\subset X$ a compact subset and $R>0$. Let $\alpha\geq 0$.
Then there exists a constant $C>0$ such that the following holds:
\begin{eqnarray*}
 \int_{X \cap \left(B_{r_2}(z)\setminus \o{B_{r_1}(z)}\right)}
\frac{dV_X(\zeta)}{|\zeta-z|^\alpha} \leq C \left\{
\begin{array}{ll}
r_2^{2n-\alpha} & \ ,\ \alpha<2n,\\
1+|\log r_1| & \ ,\ \alpha=2n,\\
r_1^{2n-\alpha} & \ ,\ \alpha>2n,
\end{array}\right.
\end{eqnarray*}
for all $z\in K$ and $0<r_1 \leq r_2 \leq R$.
\end{lem}

\begin{lem}[\cite{LR2}, Lemma 2.3]\label{lem:integral-log}
    Let $X$ and $K$ be as in Lemma~\ref{lem:estimate2}.
    Then
    \begin{eqnarray*}
        \int_{X\cap B_{1/2}(z)} \frac{dV_X(\zeta)}{|\zeta-z|^{2n} \log^2|\zeta-z|} &\lesssim& 1,  \quad  z\in K.
    \end{eqnarray*}
\end{lem}


\begin{lem}[\cite{LR2}, Lemma 2.5]\label{lem:estimate3}
Let $X\subset \C^N$ be an analytic variety of pure dimension $n$, $D\subset\subset X$ relatively compact and $0 \leq \alpha,\beta <2n$.
Then there exists a constant $C>0$ such that the following holds:
\begin{eqnarray*}
\int_D \frac{dV_X(\zeta)}{|\zeta-z|^\alpha |\zeta-w|^\beta}
\leq C\left\{
\begin{array}{ll}
1 & \ ,\ \alpha+\beta<2n,\\
\log |z - w| &\ ,\ \alpha+\beta=2n,\\
|z - w|^{2n-\alpha-\beta} &\ ,\ \alpha+\beta>2n,
\end{array}\right.
\end{eqnarray*}
for all $z, w \in X$ with $z\neq w$.
\end{lem}

\begin{lem}[\cite{LR2}, Lemma 2.7]\label{lem:estimate8}
Let $X\subset \C^N$ be an analytic variety of pure dimension $n$, $K\subset X$ a
compact subset and $R>0$. Let $0 \leq \alpha < 2n$.
Then there exists a constant $C>0$ such that:
\begin{eqnarray*}
\int_{X \cap B_{r}(z)}
\frac{dV_X(\zeta)}{|\zeta-w|^\alpha} \leq C
r^{2n-\alpha}
\end{eqnarray*}
for all $z\in K$, $w \in X$ and $0\leq r \leq R$.
\end{lem}




\medskip
\noindent{\bf Acknowledgments.}
This research was supported by the Deutsche Forschungsgemeinschaft (DFG, German Research Foundation),
grant RU 1474/2 within DFG's Emmy Noether Programme.
The first, second, and last author were partially supported by the Swedish Research Council.
We would like to thank the anonymous referee for valuable comments regarding the presentation of the article.


\begin{thebibliography}{99999}



\bibitem[AZ1]{AZ1}{\sc F.\ Acosta, E.\ S.\ Zeron},
H\"older estimates for the $\dq$-equation on surfaces with simple singularities,
{\em Bol. Soc. Mat. Mexicana} {\bf 12} (2006), no. 2, 193--204.

\bibitem[AZ2]{AZ2}{\sc F.\ Acosta, E.\ S.\ Zeron},
H\"older estimates for the $\dq$-equation on surfaces with singularities of the type $E_6$ and $E_7$,
{\em Bol. Soc. Mat. Mexicana} {\bf 13} (2007), no. 1, 73--86.


\bibitem[AS1]{AS1} {\sc M.\ Andersson, H.\ Samuelsson},
Weighted Koppelman formulas and the $\dq$-equation on an analytic space,
{\em J. Funct. Anal.} {\bf 261} (2011), 777--802.

\bibitem[AS2]{AS} {\sc M.\ Andersson, H.\ Samuelsson},
A Dolbeault--Grothendieck lemma on complex spaces via Koppelman formulas,
{\em Invent. Math.} {\bf 190} (2012), no. 2, 261--297.


\bibitem[AW]{AW} {\sc M.\ Andersson, E.\ Wulcan},
Regularity of pseudomeromorphic currents,
Preprint, 1703.01247 [math.CV].


\bibitem[BPV]{BPV} {\sc W.\ Barth, C. Peters, A. Van de Ven},
{\em Compact complex surfaces}, Ergebnisse der Mathematik und ihrer Grenzgebiete, 3. Folge, {\bf 4}. Springer-Verlag, Berlin, 1984.

\bibitem[D1]{De}{\sc J.-P.\ Demailly}, {\em Complex Analytic and Differential Geometry},
online book, available at {\sf www-fourier.ujf-grenoble.fr/$\sim$demailly/manuscripts/agbook.pdf}, Institut Fourier, Grenoble.

\bibitem[D2]{Du}{\sc A.\ H.\ Durfee}, Fifteen characterizations of rational double points and simple critical points,
{\em Enseign. Math. (2)} {\bf 25} (1979), 131--163.


\bibitem[FOV]{FOV}{\sc J. E. Forn{\ae}ss, N. {\O}vrelid, S. Vassiliadou},
Local $L^2$ results for $\overline\partial$: the isolated singularities case, {\em Internat. J. Math.}
{\bf 16} (2005), no. 4, 387--418.

\bibitem[HP]{HP}{\sc G. M. Henkin, P. L. Polyakov},
The Grothendieck-Dolbeault lemma for complete intersections,
{\em C. R. Acad. Sci. Paris S\'er. I Math.} {\bf 308} (1989), no. 13, 405--409.

\bibitem[I]{Ishii}{\sc S. Ishii},
{\em Introduction to singularities}, Springer, Tokyo, 2014.

\bibitem[LM]{LiMi}{\sc I.\ Lieb, J.\ Michel},
{\em The Cauchy-Riemann complex},
Integral formulae and Neumann problem. Aspects of Mathematics, E34. Friedr. Vieweg \& Sohn, Braunschweig, 2002.

\bibitem[LT]{LT}{\sc C.\ Laurent-Thi\'ebaut},
{\em Holomorphic function theory in several variables},
Universitext. Springer-Verlag London, Ltd., London; EDP Sciences, Les Ulis, 2011.

\bibitem[LR1]{LR}{\sc R.\ L\"ark\"ang, J.\ Ruppenthal},
Koppelman formulas on the $A_1$-singularity,
{\em J. Math. Anal. Appl.} {\bf 437} (2016), 214--240.

\bibitem[LR2]{LR2}{\sc R.\ L\"ark\"ang, J.\ Ruppenthal},
Koppelman formulas on affine cones over smooth projective complete intersections,
{\em Indiana Univ. Math. J.} {\bf 67} (2018), no. 2, 753--780. 

\bibitem[M]{M}{\sc E. J. McShane},
Extension of range of functions
{\em Bull.\ Amer.\ Math.\ Soc.},
{\bf 40} (1934), 837--842.


\bibitem[OR]{OR}{\sc N.\ {\O}vrelid, J.\ Ruppenthal},
$L^2$-properties of the $\dq$ and the $\dq$-Neumann operator on spaces with isolated singularities,
{\em Math. Ann.} {\bf 359} (2014), 803--838.


\bibitem[OV]{OV}{\sc N.\ {\O}vrelid, S.\ Vassiliadou},
$L^2$-$\dq$-cohomology groups of some singular complex spaces,
{\em Invent. Math.} {\bf 192} (2013), no. 2, 413-458.

\bibitem[P]{P}{\sc W. Pardon}, The $L^2$-$\dq$-cohomology of an algebraic surface,
{\em Topology} {\bf 28} (1989), no. 2, 171--195.

\bibitem[PS]{PS}{\sc W.\ Pardon, M.\ Stern},
$L^2$-$\dq$-cohomology of complex projective varieties, {\em J. Amer. Math. Soc.} {\bf 4} (1991), no. 3, 603--621.

\bibitem[Ra]{Ra}{\sc R.\ M.\ Range}, {\em Holomorphic functions and integral representations in several complex variables},
Graduate Texts in Mathematics, {\bf 108}. Springer-Verlag, New York, 1986.

\bibitem[R1]{RuDipl}{\sc J.\ Ruppenthal},
Zur Regularit\"at der Cauchy-Riemannschen Differentialgleichungen auf komplexen Kurven,
Diplomarbeit, University of Bonn, 2003.

\bibitem[R2]{RuThesis}{\sc J.\ Ruppenthal},
{\em Zur Regularit\"at der Cauchy-Riemannschen Differentialgleichungen auf komplexen R\"aumen},
Dissertation, Rheinische Friedrich-Wilhelms-Universit\"at Bonn, Bonn, 2006. Bonner Mathematische Schriften, {\bf 380}.
Universit\"at Bonn, Mathematisches Institut, Bonn, 2006.

\bibitem[R3]{RMatZ}{\sc J.\ Ruppenthal},
The $\overline\partial$-equation on homogeneous varieties with an isolated singularity,
{\em Math. Z.} {\bf 263} (2009), 447--472.

\bibitem[R4]{RDuke}{\sc J.\ Ruppenthal},
$L^2$-theory for the $\dq$-operator on compact complex spaces,
{\em Duke Math. J.} {\bf 163} (2014), 2887--2934.

\bibitem[R5]{RSerre}{\sc J.\ Ruppenthal},
$L^2$-Serre duality on singular complex spaces and rational singularities,
{\em Int. Math. Res. Not. IMRN} {\bf 23} (2018), 7198--7240.

\bibitem[RZ]{RZ2}{\sc J.\ Ruppenthal, E.\ Zeron},
An explicit $\overline\partial$-integration formula for weighted homogeneous varieties II. Forms of higher degree,
{\em Michigan Math. J.} {\bf  59} (2010), no. 2, 283--295.

\bibitem[S]{S}{\sc Y.-T.\ Siu},
Analytic sheaf cohomology groups of dimension $n$ of $n$-dimensional noncompact complex manifolds,
{\em Pacific J. Math.} {\bf 28} (1969), 407--411.





\end{thebibliography}
\end{document}